\documentclass[a4paper,11pt,reqno]{amsart}
\usepackage[a4paper,top=3cm,bottom=2cm,left=3cm,right=3cm,marginparwidth=1.75cm]{geometry}
\usepackage[english]{babel}
\usepackage{amsfonts, amsmath, amssymb, amsthm}
\usepackage{bm}
\usepackage{tikz-cd}
\usetikzlibrary{arrows}

\usepackage{hyperref} 
\usepackage[OT2,T1]{fontenc}
\DeclareSymbolFont{cyrletters}{OT2}{wncyr}{m}{n}
\DeclareMathSymbol{\Beh}{\mathalpha}{cyrletters}{"42}
\DeclareMathSymbol{\beh}{\mathalpha}{cyrletters}{"62}
\DeclareMathSymbol{\Sha}{\mathalpha}{cyrletters}{"58}

\newcommand{\N}{\mathbb{N}}
\newcommand{\Z}{\mathbb{Z}}
\newcommand{\Q}{\mathbb{Q}}
\newcommand{\R}{\mathbb{R}}
\newcommand{\C}{\mathbb{C}}
\newcommand{\F}{\mathbb{F}}
\newcommand{\G}{\mathbb{G}}
\newcommand{\A}{\mathbb{A}}

\DeclareMathOperator{\Hom}{Hom}
\DeclareMathOperator{\Pic}{Pic}
\DeclareMathOperator{\Gal}{Gal}
\DeclareMathOperator{\Proj}{\mathbb{P}}
\DeclareMathOperator{\Spec}{Spec}
\DeclareMathOperator{\Br}{Br}

\DeclareMathOperator{\Aut}{Aut}

\newtheorem{theorem}{Theorem}[section]
\newtheorem{proposition}[theorem]{Proposition}
\newtheorem{corollary}[theorem]{Corollary}
\newtheorem{lemma}[theorem]{Lemma}

\theoremstyle{definition}
\newtheorem{definition}[theorem]{Definition}
\newtheorem{example}[theorem]{Example}
\newtheorem{remark}[theorem]{Remark}
\newtheorem{construction}[theorem]{Construction}

\title{The Brauer-Manin obstruction for stacky curves}
\author{Tim Santens}
\begin{document}
\address{University of Cambridge, DPMMS, Centre for Mathematical Sciences, Wilberforce Road, Cambridge, CB3 0WB, UK}
\email{ts996@cam.ac.uk}
\thanks{The author is supported by FWO-Vlaanderen (Research Foundation-Flanders) with grant number 11I0621N}
\begin{abstract}
    We show that the Brauer-Manin obstruction is the only obstruction to strong approximation for all tame stacky curves over global fields with finite abelian fundamental groups. This includes all tame stacky curves of genus $g = \frac{1}{2}$, thus explaining a recent counterexample to the Hasse principle of Bhargava-Poonen. We will furthermore show that the elementary obstruction is the only obstruction to the integral Hasse principle for smooth proper integral models of tame stacky curves of genus $g < 1$. We then compute the Brauer-Manin obstruction for smooth proper integral models of stacky curves of genus $\frac{1}{2}$.
\end{abstract}
\maketitle
\tableofcontents
\section{Introduction}
Let $k$ be a global field (a finite extension of $\Q$ or $\F_p(T)$). If $X$ is a variety over $k$ then a necessary condition for $X(k)$ to be non-empty is that $X(k_v) \neq \emptyset$ for all places $v$ of $k$. If this is also sufficient we say that $X$ satisfies the \emph{Hasse principle}.

Recently there has been significant interest in considering arithmetic questions such as whether $X$ satisfies the Hasse principle in the case when $X$ is an algebraic stack instead of a variety. In particular, Bhargava and Poonen \cite{Bhargava2020Stacky} have shown that the integral Hasse principle holds for a \emph{stacky curve} if the genus is $g < \frac{1}{2}$. They also gave an explicit example of a stacky curve of genus $g = \frac{1}{2}$ which fails the integral Hasse principle.

If $X$ has a dense open subscheme, for example if $X$ is a stacky curve, then $X(k)$ is almost equal to the set of rational points $X_{\text{coarse}}(k)$ on its coarse moduli space $X_{\text{coarse}}$, which is a scheme. We will instead look at \emph{integral points}.

Let $U$ be a finite type Dedekind scheme with fraction field $k$, this is either a smooth curve over a finite field or the spectrum of the ring $\mathcal{O}_{k, S}$ of $S$-integers for a finite set of places $S$ of a number field $k$. An \emph{integral model} $\mathcal{X}$ of $X$ is a separated finite type algebraic stack $\mathcal{X}$ over $U$ equipped with an isomorphism between the generic fiber $\mathcal{X}_k$ and $X$. The set of integral points on $\mathcal{X}$ is $\mathcal{X}(U)$. We remark that if $\mathcal{X}$ is proper over $U$ then we do not necessarily have $\mathcal{X}(U) = X(k)$, unlike the case when $\mathcal{X}$ is a scheme.

A necessary condition for the non-emptiness of $\mathcal{X}(U)$ is that $\mathcal{X}(\mathcal{O}_v) \neq \emptyset$ for all places $v$ of $k$ which are defined by points of $U$ and $X(k_v) \neq \emptyset$ for all other places. If this condition is also sufficient then we say that $\mathcal{X}$ satisfies the \emph{integral Hasse principle}. 

Although the study of the arithmetic of stacky curves is novel the study of their integral points is not. Indeed, solutions of the generalized Fermat equation $A X^p + B Y^q = C Z^r$ with $X,Y,Z$ coprime can be interpreted as points on a certain stacky curve. This viewpoint is exploited by Poonen, Schaefer and Stoll in \cite{Poonen2007Twists} to study the solutions of $X^2 + Y^3 = Z^7$. Furthermore, Darmon and Granville's work \cite{Darmon1995Equations,Darmon1997Faltings} can be interpreted as the statement that Faltings's theorem holds for stacky curves, i.e. stacky curves of genus $g > 1$ have only a finite number of integral points. We also mention work of Beukers \cite{Beukers1998Diophantine} which can be done with less explicit computations using stacky methods, the parametrisations in his work correspond to torsors under the geometric fundamental group.

Recently Mitankin, Nakahara and Streeter \cite{Mitankin2022Semi-Integral} studied the Brauer-Manin obstruction for Darmon points on certain quadric orbifolds. The Darmon points on a variety are by definition almost the same as the integral points on a certain root stack, and their orbifold Brauer group is the Brauer group of the corresponding root stack \cite[Rem. 3.12]{Mitankin2022Semi-Integral}. The Brauer-Manin obstruction on such a quadric orbifold behaves rather differently from the Brauer-Manin obstruction on stacky curves, the Brauer group modulo constants is finite for quadric orbifolds but it is generally infinite for stacky curves.

A single stack can have many different integral models. But we can describe the subset $\mathcal{X}(U) \subset X(k)$ using purely local conditions, it is the pre-image of the integral adelic points $\mathcal{X}(\A_U)$ under the map $X(k) \to X(\A_k)$ \cite[Lemma 13]{Bhargava2020Stacky}. To study the Hasse principle on all integral models it thus suffices to study the image of $X(k)$ in $X(\A_k)$.

Two important explanations for why a variety might fail the Hasse principle are the \emph{Brauer-Manin obstruction} and the \emph{finite descent obstruction} \cite{Stoll2007FiniteDescent}. The first main result of this paper is.
\begin{theorem}
Let $X$ be a tame stacky curve over $k$ of genus $g < 1$.
\begin{enumerate}
    \item The finite descent obstruction is the only obstruction to strong approximation. More precisely there exists a torsor $\pi:Y \to X$ under a finite \'etale group scheme such that descent along $\pi$ is the only obstruction to strong approximation, i.e. the image of $X(k)$ is dense in $X(\A_k)^{\pi} \subset X(\A_k)$.
    \item If $X$ has a finite abelian geometric fundamental group then the only obstruction to strong approximation is the Brauer-Manin obstruction, i.e. the image of $X(k)$ is dense in the Brauer-Manin set $X(\A_k)^{\Br} \subset X(\A_k)$.
\end{enumerate}
\label{Brauer-Manin obstruction is only obstruction to strong approximation}
\end{theorem}

We will recall the definition of $X(\A_k)^{\pi}$ and $X(\mathbb{A}_k)^{\Br}$ in \S 4 and the topology we consider on $X(\A_k)$ is the one constructed by Christensen \cite{Christensen2020Topology}. 

If $X$ is a geometrically simply connected tame stacky curve then $\Br X = \Br k$ so there is no Brauer-Manin obstruction.
\begin{theorem}
Geometrically simply connected tame stacky curves $X$ over $k$ satisfy strong approximation, i.e. the image of $X(k)$ in $X(\A_k)$ is dense.
\label{Strong approximation for simply connected stacky curves}
\end{theorem}
A weaker form of this result is due to Christensen \cite[Thm 1.0.1]{Christensen2020Topology}. He only considers stacky curves of genus $g < \frac{1}{2}$ and shows the weaker statement that strong approximation holds after removing a place in $\A_k$.

All tame stacky curves with genus $g < \frac{1}{2}$ are geometrically simply connected due to the Riemann-Hurwitz formula. The above theorem thus gives a geometric explanation of why the minimal genus of a stacky curve for which the integral Hasse principle fails is $\frac{1}{2}$.

If $\mathcal{X}$ is an integral model of $X$ then $\mathcal{X}(\A_{U})$ is an open subset of $X(\A_k)$ so we deduce.
\begin{corollary}
Let $\mathcal{X}$ be an integral model of a tame stacky curve $X$ with finite abelian geometric fundamental group. The Brauer-Manin obstruction is the only obstruction to the integral Hasse principle for $\mathcal{X}$.
\label{Brauer-Manin obstruction is only obstruction to the integral Hasse principle for stacky curves with abelian fundamental group}
\end{corollary}

The Brauer group of a stacky curve is in general infinite so computing the set $\mathcal{X}(\A_U)^{\Br}$ is rather difficult in practice. But if $\mathcal{X} \to U$ is smooth and proper then it turns out that there is a finite subquotient $\Beh(\mathcal{X})$ of $\Br X$ which is still able to explain all failures of the integral Hasse principle and for which the corresponding Brauer-Manin set $\mathcal{X}(\A_U)^{\Beh(\mathcal{X})}$ is easier to compute. The group $\Beh (\mathcal{X})$ of \emph{locally constant algebras} is defined in \S 4.

The obstruction coming from $\Beh(X)$ over a number field is known \cite[Prop. 3.3.2, 3.6.4]{Colliot1987Descente} to be related to the vanishing of the \emph{elementary obstruction} \cite[Def. 2.2.1]{Colliot1987Descente} for varieties. We will extend the definition of the elementary obstruction to integral points using a construction due to Harari and Skorobogatov \cite[\S 1]{Harari2013DescentOpen} and prove that the relation with the Brauer-Manin obstruction extends to the integral case and for stacks. In general we have the following.
\begin{theorem}
Let $\mathcal{X} \to U$ be a proper smooth map of Deligne-Mumford stacks such that every fiber is a tame stacky curve. Let $X$ be the generic fiber and assume it has genus $g < 1$.
\begin{enumerate}
    \item The elementary obstruction is the only obstruction to the existence of an integral point for $\mathcal{X}$.
    \label{The elementary obstruction is the only obstruction: part 1}
    \item If there is at least one non-archimedean place  of $k$ which is not induced by a point of $U$ then $\mathcal{X}(\A_U)^{\Beh(\mathcal{X})} \neq \emptyset$ implies $\mathcal{X}(U) \neq \emptyset$.
    \label{The Brauer-Manin obstruction coming from beh is the only one}
    \item In particular, if $X$ has a signature\footnote{The signature is defined in \S\ref{sec:stacky_curves}} which is not equal to  either $(0)$ or $(0; n, m)$ with $\gcd(n, m) > 1$ then $\mathcal{X}(U) \neq \emptyset.$
    \label{Only stacky curves with signature (0; n, m) can have an obstruction}
\end{enumerate}
\label{The elementary obstruction is the only obstruction}
\end{theorem}
The assumption that at least one non-archimedean place of $k$ does not come from $U$ is necessary, see Example \ref{ex:Root stacks on Hirzebruch surfaces}.

\begin{remark}
    In a previous version of this paper it was claimed that all smooth proper models of stacky curves of signature $(0;n, m)$ with $n \neq m$ have integral points. This is false as shown by Example \ref{ex:Root stacks on Hirzebruch surfaces}.

    The proof of Theorem \ref{The elementary obstruction is the only obstruction}(3) is also more involved than in the previous version. The hardest case turns out to be when the signature is $(0; 2,2,2)$. The proof for this case in \S \ref{sec:(0;2,2,2)_case} was obtained in a conversation of the author with GPT 6 Astra, see the AI use statement for details.
\end{remark}

This theorem allows us to give a simple criterion for when a smooth proper integral model of a stacky curve of genus $\frac{1}{2}$ has an integral point in Theorem \ref{Criterion for a Brauer-Manin obstruction coming from a single algebra}.
\subsection*{Structure}
In \S 2 we recall some preliminaries on the analytic topology on stacks and twisting by torsors.

We start \S 3 by recalling the definition of a stacky curve and their invariants. We show that smooth proper integral models of stacky curves have a simple global structure. This is used to prove the simplest cases of Theorem \ref{The elementary obstruction is the only obstruction}.

We continue \S 3 by proving Theorem \ref{Strong approximation for simply connected stacky curves} and then apply it and the descent theory of Colliot-Th\'el\`ene and Sansuc \cite{Colliot1987Descente} to prove Theorem \ref{Brauer-Manin obstruction is only obstruction to strong approximation} in \S 4.

We will then compute the Brauer group of a stacky curve in \S 5 and show that it is generated by certain cup products.

In \S 6 we introduce the elementary obstruction to integral points and use it to prove the first and last part of Theorem \ref{The elementary obstruction is the only obstruction}. We then show that this implies the second part of Theorem \ref{The elementary obstruction is the only obstruction}, analogously to \cite[Cor. 3.6]{Harari2013DescentOpen}.

In \S 7 we apply Theorem \ref{The elementary obstruction is the only obstruction} to give a simple necessary and sufficient criterion for a genus $\frac{1}{2}$ curve over $\Q$ to fail the integral Hasse principle. We finish by using this criterion to reinterpret the example of Bhargava-Poonen.

\subsection*{Conventions}
For any stacks $X, S$ we will use $X(S)$ to denote the set of equivalence classes of morphisms $S\to X$ and for a ring $R$ we write $X(R) := X( \Spec R)$.

Fix for every field $k$ a separable closure $k_s$ and an algebraic closure $\overline{k}$.

If $A$ is a commutative group then we write $A^{\vee} := \Hom(A, \Q/ \Z)$ for its dual. We let $A[n] \subset A$ be the subgroup of $n$-torsion elements.

All cohomology/homotopy groups are \'etale cohomology/homotopy groups. If $G$ is a smooth group scheme over $X$ then $H^1(X, G)$ classifies $G$-torsors over $X$. We will freely use this identification.
\subsection*{Acknowledgments}
Firstly I would like to thank my advisor Arne Smeets for his encouragement and help. I have also benefited from helpful conversations with Frits Beukers and Siddharth Mathur. I would like to thank Daniel Loughran, Bjorn Poonen and John Voight for their useful comments and help in improving the exposition.
\subsection*{AI use statement}
GPT Astra was used to generate the proof of Proposition \ref{prop:signature_(2,2,2)}. In particular, it provided the relation to the theta group of an elliptic curve and Lemma \ref{lem:existence_of_good_twist}. The author chose to give a slightly different presentation of the main ideas to make the argument clearer. The same model was also used to audit the paper for mistakes.
\section{Preliminaries}
\subsection{The analytic topology for points on stacks}
Let $R$ be a topological ring which is \emph{essentially analytic} \cite[Def. 5.0.8]{Christensen2020Topology}. Examples of essentially analytic $R$ are local fields $k$ and their rings of integers $\mathcal{O}_k$.

Christensen \cite{Christensen2020Topology} defines a topology on $X(R)$ for any finitely presented separated algebraic stack $X$ over $R$ which extends the usual analytic topology for $R$-points on schemes and has all the expected properties. 

Let $U$ be a Dedekind scheme whose field of fractions $k$ is a global field and $S$ the set of places of $k$ which do not come from points of $U$. Let $\A_k$ be the ring of adeles and $\A_U$ the subring of integral adeles. Let $\mathcal{X}$ be a separated finite type Deligne Mumford stack over $U$ with generic fiber $X$. We give $\mathcal{X}(\A_U) = \prod_{v \in S} X(k_v) \times \prod_{v \not \in S} \mathcal{X}(\mathcal{O}_v)$ the product topology. Christensen also defines a topology on $X(\A_k)$ such that we have a homeomorphism of topological spaces \cite[Prop. 13.0.2]{Christensen2020Topology}:
\begin{equation*}
    X(\A_k) \cong \prod_{v \in S} X(k_v) \times \sideset{}{'}\prod_{v \not \in S}(X(k_v), \mathcal{X}(\mathcal{O}_v)). 
\end{equation*}
The right-hand side is a restricted product. In particular, he shows that the right-hand side is independent of the integral model $\mathcal{X}$.

\begin{remark}
The map $X(k) \to X(\mathbb{A}_k)$ is not necessarily injective when $X$ is a Deligne-Mumford stack, a subtlety which does not happen if $X$ is a scheme. For example, if $X := B \mu_8$ is the classifying stack of $\mu_8$ then the above map is $k^{\times}/ k^{\times 8} \to \A_k^{\times}/ \A_k^{\times 8}$, which is not injective if $k = \Q(\sqrt{7})$, see \cite[Thm. 9.1.11]{Neukirch2008Cohomology}. In practice this seems unimportant.
\end{remark}
\subsection{Twists of torsors} The following is explained in more detail in \cite[\S 2.2]{Skorobogatov2001Torsors}. Let $S$ be a scheme and $G$ a smooth group scheme over $S$. A right $G$-torsor $\mathfrak{a}:T \to S$ is also a left $G_{\mathfrak{a}}$-torsor \cite[p. 13]{Skorobogatov2001Torsors}, where $G_{\mathfrak{a}}$ is the result of twisting $G$ by the cocycle $\mathfrak{a}$ via the conjugation action $G \to \text{Aut}(G):  g \to  (h \to g h g^{-1})$. If $G$ is commutative then $G_{\mathfrak{a}} = G$. We also have the inverse torsor $\mathfrak{a}^{-1}: T \to S$ which has the same underlying scheme as $T$ and which is a right $G_{\mathfrak{a}}$-torsor via the action $t \cdot_{\mathfrak{a}^{-1}} g = g^{-1} \cdot_{\mathfrak{a}} t$. It is analogously a left $G$-torsor.

Let $X$ be an algebraic stack over $S$ and $\pi:Y \to X$ a right $G$-torsor. The \emph{twisted torsor} \cite[p.20, Ex. 2]{Skorobogatov2001Torsors} $\pi_{\mathfrak{a}}: Y_{\mathfrak{a}} \to X$ is the right $G_{\mathfrak{a}}$-torsor given by the quotient stack $[T \times_{\mathfrak{a}^{-1}, S} Y/G]$. The action of $G$ is the diagonal one 
\begin{equation*}
    G \times_S Y \times_S T \to  Y\times_S T: (g; y, t) \to (y \cdot g, g \cdot_{\mathfrak{a}} t) = (y \cdot g, t  \cdot_{\mathfrak{a}}  g ).
\end{equation*}
If $G$ is commutative then the cohomology class of $\pi_{\mathfrak{a}}: Y_{\mathfrak{a}} \to X$ is equal to $[Y] - \mathfrak{a} \in H^1(X, G)$.

We can use torsors to study points on $X$ using the following lemma.
\begin{lemma}
The images $\pi_{\mathfrak{a}}(Y_{\mathfrak{a}}(S))$ of the twists $\pi_{\mathfrak{a}}: Y_{\mathfrak{a}} \to X$  are pairwise disjoint and 
\begin{equation*}
    X(S) = \coprod_{\mathfrak{a} \in H^1(S, G)}\pi_{\mathfrak{a}}(Y_{\mathfrak{a}}(S)).
\end{equation*}
\label{Points on the base and on the torsor}
\end{lemma}
\begin{proof}
If $X$ is a variety then this is \cite[Thm 8.4.1]{Poonen2017Rational}. The proof for stacks is completely analogous.
\end{proof}
\section{Stacky curves}\label{sec:stacky_curves}
\subsection{Stacky curves}
Let $k$ be a field of characteristic $p$ (possibly 0). We define a \emph{stacky curve} over $k$ to be a proper smooth geometrically connected Deligne-Mumford stack of dimension $1$ over $k$ which contains a dense open subscheme. We call a stacky curve \emph{tame} if $p$ does not divide the order of any stabilizer. We refer to \cite[\S 5]{Voight2022Canonical} for basic facts about stacky curves.

For a tame stacky curve $X$ we let $X_{\text{coarse}}$ be its coarse moduli space, which is always a smooth proper curve. The \emph{coarse genus} $g_{\text{coarse}}$ is defined as the genus of $X_{\text{coarse}}$. Let $\mathcal{P} \subset X_{\text{coarse}}$ be the closed subscheme on which the canonical map $X \to X_{\text{coarse}}$ is not an isomorphism, we will call this the \emph{stacky locus}. Let $1 < e_1, \cdots, e_r$ be the orders of the stabilizers of the finite set of points lying above the points of $\mathcal{P}(\overline{k}) = \mathcal{P}(k_s)$, which we call the \emph{stacky points} (the equality follows from \cite[Lemma 7]{Bhargava2020Stacky}). The \emph{signature} of a stacky curve is the tuple $(g_{\text{coarse}}; e_1, \cdots, e_r)$, it is well-defined up to a permutation of the $e_i$. We will also write $(g; 1, \cdots, 1,  e_1, \cdots, e_r) := (g; e_1, \cdots, e_r)$. The \emph{genus} of a tame stacky curve $X$ is 
\begin{equation}
    g := g_{\text{coarse}} + \frac{1}{2}\left( \frac{e_1 -1}{e_1} + \cdots + \frac{e_r -1}{e_r}\right).
    \label{Genus of stacky curve}
\end{equation}
The Riemann-Hurwitz formula holds for stacky curves with this choice of genus. This follows from \cite[Lemma 5.5.3, Prop. 5.5.6]{Voight2022Canonical}. If $g < 1$, i.e. if $X$ is Fano, then $g_{\text{coarse}} = 0$ and $r \leq 3$.

A tame stacky curve is completely determined by its coarse moduli space, its stacky points and the numbers $e_i$ by \cite[Lemma 5.3.10]{Voight2022Canonical}. In particular, every tame stacky curve is a root stack \cite[Def. 2.2.1]{Cadman2007Tangency} over its coarse moduli space $X_{\text{coarse}}$. This implies that dominant maps between stacky curves are the same as dominant maps between their coarse moduli spaces satisfying certain ramification conditions.
\subsubsection{Relative stacky curves}
A \emph{relative (tame) stacky curve} over a scheme $S$ is a finitely presented smooth proper morphism $\mathcal{X} \to S$ of Deligne-Mumford stacks whose fibers are (tame) stacky curves. It is claimed in \cite[Ex. 11.2.2]{Voight2022Canonical} that relative stacky curves can have certain pathologies, but this is false. In their first example the structure morphism is not well-defined and it is not smooth in their second example. We will prove that all relative tame stacky curves are \emph{twisted}, in the sense of \cite[Def. 11.2.1]{Voight2022Canonical}. 

Let $X$ be a Deligne-Mumford stack, $D \subset X$ a Cartier divisor and $n \in \N$. We will use the notation $X[\sqrt[n]{D}]$ for the root stack which is rooted at $D$, i.e. what is denoted by $X_{D, n}$ in \cite[Def. 2.2.4]{Cadman2007Tangency}.
\begin{proposition}\label{Prop:Global structure stacky curves}
Let $\mathcal{X} \to S$ be a relative tame stacky curve. The map $\mathcal{X}_{\emph{coarse}} \to S$ is smooth.

The stack $\mathcal{X}$ is isomorphic to a root stack over $\mathcal{X}_{\emph{coarse}}$ rooted at relative divisors which are pairwise disjoint and \'etale over $S$.

\end{proposition}
\begin{proof}
The map $\mathcal{X}_{\text{coarse}} \to S$ is flat since $\mathcal{X} \to S$ is flat by \cite[Cor. 3.3]{Abramovich2008Tame}. By loc. cit. we also have for any point $s \in S$ an isomorphism $(\mathcal{X}_{\text{coarse}})_s \cong (\mathcal{X}_s)_{\text{coarse}}$. The latter object is the coarse moduli space of a smooth stacky curve and thus smooth. This implies that $\mathcal{X}_{\text{coarse}} \to S$ is smooth.

Let $U \subset \mathcal{X}_{\text{coarse}}$ be the maximal open subscheme on which the map $p: \mathcal{X} \to \mathcal{X}_{\text{coarse}}$ is an isomorphism. Let $D := \mathcal{X}_{\text{coarse}} \setminus U$, $D_1, \cdots, D_k \subset D$ be the connected components of $D$ and $E  := \mathcal{X} \setminus p^{-1}(U)$. 

We now claim the following: the divisors $D$ and $E$ are Cartier and $D \to S$ is finite \'etale. For each $1 \leq i \leq k$ there exists an $n_i \in \N$ and a connected component $E_i \subset E$ such that $p^{-1}(D_i) = n_i E_i$.

All of these statements can be checked \'etale locally on $\mathcal{X}_{\text{coarse}}$ so by \cite[Prop. 2.2]{Olsson2007Twisted} we may assume that the map $p$ is the following projection map 
\begin{equation}
    [\Spec \mathcal{O}_{\mathcal{X}_{\text{coarse}}}[z]/(z^n - \pi)/ \mu_n] \to \mathcal{X}_{\text{coarse}}.
    \label{Local description of relative stacky curve}
\end{equation}
Here $n \in \N$, $\zeta \in \mu_n$ acts via  $z \to \zeta z$ and $\pi \in \mathcal{O}_{\mathcal{X}_{\text{coarse}}}$ is such that the induced map $\mathcal{X}_{\text{coarse}} \to \A^1_S$ is \'etale. For this projection we have $D = \{\pi = 0 \}$, $E = \{z = 0\}$ so all the above claims hold.

The map $p$ lifts to a map $\mathcal{X} \to \mathcal{X}_{\text{coarse}}[\sqrt[n_1]{D_1}] \cdots [\sqrt[n_k]{D_k}]$ to the iterated root stack by the universal property of root stacks. 

The map \eqref{Local description of relative stacky curve} is the structure map of the root stack $\mathcal{O}_{{\mathcal{X}_{\text{coarse}}}}[\sqrt[n]{\{ \pi = 0\}}]$ by \cite[Ex. 2.4.1]{Cadman2007Tangency}. We deduce that $\mathcal{X} \to \mathcal{X}_{\text{coarse}}[\sqrt[n_1]{D_1}] \cdots [\sqrt[n_k]{D_k}]$ is \'etale locally an isomorphism, hence an isomorphism.
\end{proof}
We say that a relative stacky curve has signature $(g_{\text{coarse}}; e_1, \cdots, e_r)$ if all its fibers have signature $(g_{\text{coarse}}; e_1, \cdots, e_r)$. An immediate corollary of the Proposition is that this is a reasonable notion.
\begin{corollary}
    If $\mathcal{X} \to S$ is a relative tame stacky curve and $S$ is connected then it has a signature.
\end{corollary}
\begin{proof}
    The map $\mathcal{X}_{\text{coarse}} \to S$ is smooth and proper so the coarse genus of the fibers is locally constant. The $e_i$ are constant since the number of points in each geometric fiber of the finite \'etale map $D_i \to S$ in Proposition \ref{Prop:Global structure stacky curves} is locally constant.
\end{proof}

The proposition also allows us to describe the structure of relative tame stacky curves of coarse genus $0$ \'etale locally on the base.
\begin{corollary}\label{cor:Stacky curves are locally split}
    Let $\mathcal{X} \to S$ be a relative tame stacky curve with signature $(0; e_1, \cdots, e_k)$. Then for every $s \in S$ there exists an \'etale neighbourhood $U \to S$ and $D_1, \cdots, D_k \in \Proj^1(U)$ such that $\mathcal{X}_{U} \cong \Proj^1_U[\sqrt[e_1]{D_1}] \cdots [\sqrt[e_k]{D_k}]$.
\end{corollary}
\begin{proof}
    The map $\mathcal{X}_{\text{coarse}} \to S$ is a relative genus $0$ curve so is \'etale locally isomorphic to $\Proj^1$. By taking an \'etale neighbourhood $U$ of $s$ we may assume that the divisors appearing in Proposition \ref{Prop:Global structure stacky curves} are sections $U \to (\mathcal{X}_{\text{coarse}})_U$, the claim follows.
\end{proof}
We can now prove the first cases of Theorem \ref{The elementary obstruction is the only obstruction}(3).
\begin{lemma}\label{lem:Stacky curves with non-abelian fundamental group have integral points}
    Let $S$ be a connected scheme and $\mathcal{X} \to S$ a relative tame stacky curve of genus $< 1$ whose signature is neither $(0;n,m)$ for any $n,m \in \Z_{\geq 1}$ nor $(0;2,2,2)$. Then $\mathcal{X}(S) \neq \emptyset$.
\end{lemma}
\begin{proof}
    It follows from \eqref{Genus of stacky curve} that the signature of $\mathcal{X} \to S$ is equal to $(0; 2,2,n)$ for $n \geq 3$, $(0; 2,3,3), (0; 2,3,4)$ or $(0; 2, 3,5)$. That the stacky curve is tame thus implies that $2$ is invertible on $S$. 

    In all these cases the signature is equal to $(0; n, m, \ell)$ with $1 \neq n,m, \ell$ and $\ell \neq n, m$ (but potentially $n = m$). Combining Corollary \ref{cor:Stacky curves are locally split} and the fact that $\operatorname{PGL}_2$ acts transitively on $3$ points implies that $\mathcal{X}$ is \'etale locally on $S$ isomorphic to $\mathcal{Y} := \Proj^1_S[\sqrt[n]{[1:0]}][\sqrt[m]{[0:1]}][\sqrt[\ell]{[1:1]}]$, i.e. it is a twist of this stacky curve.

    The automorphism group of $\mathcal{Y} \to S$ is $\Z/2\Z$, resp. the trivial group, if $n = m$, resp. $n \neq m$. 

    Note now that $[1:-1] \in \Proj^1_S(S)$ does not intersect the divisors $[1:0], [1:1]$ and $[0:1]$ since $2$ is invertible and thus lifts to a section of $\mathcal{Y} \to S$. The automorphism group of $\mathcal{Y} \to S$ fixes this section so the section extends to any twist of $\mathcal{Y} \to S$, such as $\mathcal{X} \to S$. We conclude that $\mathcal{X}(S) \neq \emptyset$.
\end{proof}
\subsection{Fundamental groups}
The fundamental groups of stacky curves over $\C$ have been computed in terms of their signature \cite[Prop. 5.6]{Behrend2006Uniformization}. For a tame stacky curve over a general separably closed field $k$ we can replace the Van Kampen theorem by its analogue \cite[Thm. IX.5.1]{Grothendieck2003SGA1} for the \'etale topology. The analogous argument shows that the fundamental group is a quotient of the profinite completion of the group given in \cite[Prop. 5.6]{Behrend2006Uniformization} by a pro-$p$ group. 

The group given in loc. cit. is finite and coprime to $p$ if $X$ is a tame stacky curve of genus $g < 1$, so the description in \cite[Prop 5.6]{Behrend2006Uniformization} holds for all tame stacky curves of genus $< 1$ over any separably closed field. The above facts will be left implicit when we cite \cite{Behrend2006Uniformization} in what follows.
\subsubsection{Simply connected stacky curves}
Let us discuss the simply connected stacky curves. Let us first recall the definition of the weighted projective line.
\begin{definition}
Let $S$ be a stack and $n, m \in \Z_{> 0}$. The \emph{weighted projective line} $\Proj_S(n,m)$ is the quotient stack $[\A^2_S \setminus \{ 0 \}/ \G_m]$ where $\G_m$ acts on $\A^2_S$ as 
\begin{equation*}
    \G_m \times \A^2_S \to \A^2_S: (t, x, y) \to (t^n x, t^m y).
\end{equation*}
We will use the traditional notation $\Proj_R(n,m) := \Proj_{\Spec R}(n,m)$.
\end{definition}
The stack $\Proj_S(n,m)$ has a dense open subscheme if and only if $S$ has a dense open subscheme and $n$ and $m$ are coprime. If $S = \Spec k$ then the resulting stacky curve has signature $(0; n, m)$.
\begin{lemma}
The only geometrically simply connected tame stacky curves over a field $k$ of characteristic $p$ are the genus $0$ curves and the weighted projective lines $\Proj_k(n,m)$ for coprime $n,m$ with $p \nmid n,m$.
\end{lemma}
\begin{proof}
This is true over a separably closed field by \cite[Thm. 1.1]{Behrend2006Uniformization}. We may thus assume that $X$ has signature $(0;n,m)$ with $n, m$ as in the statement. If $n = m = 1$ then $X$ has genus $0$. Otherwise the action of $\Gal(k_s/k)$ on the stacky $k_s$-points has to be trivial since it preserves stabilizers, so the stacky points are $k$-points. It follows that $X \cong \Proj_k(n,m)$ by \cite[Lemma 5.3.10]{Voight2022Canonical}.
\end{proof}

We now prove Theorem \ref{Strong approximation for simply connected stacky curves}. Note that the map $\A^2 \setminus \{0 \} \to \Proj(n,m)$ is a $\G_m$-torsor so Hilbert's Theorem 90 and Lemma \ref{Points on the base and on the torsor} imply that $K^2 \setminus \{0 \} = (\A^2 \setminus \{ 0 \})(K) \to \Proj(n,m)(K)$ is surjective for any field $K$. 
\begin{theorem}
Geometrically simply connected tame stacky curves $X$ over a global field $k$ satisfy strong approximation, i.e. the image of $X(k)$ in $X(\A_k)$ is dense.
\label{Simply connected curves satisfy strong approximation}
\end{theorem}
\begin{proof}
For genus $0$ curves this is known so we may assume that $X = \Proj_k(n,m)$ with $n,m$ coprime. We will use the integral models $\Proj_{\mathcal{O}_v}(n,m)$ to describe $\Proj_k(n,m)(\A_k)$ as a restricted product.

Let $S$ be a finite set of places of $k$ containing all the archimedean ones. For every $v \in S$ let $U_v \subseteq \Proj_{k}(n,m)(k_v)$ be a non-empty open subset. Our goal is to find a point $x \in \Proj_{k}(n,m)(k)$ such that $x \in U_v$ for all $v \in S$ and $x \in \Proj_{\mathcal{O}_v}(n,m)(\mathcal{O}_{v})$ for all $v \not \in S$.

Let $\pi: \A^2_{k} \setminus \{0 \} \to \Proj_{k}(n,m)$ be the $\G_m$-torsor. We will find $(y_1,y_2) \in (\A^2_{k} \setminus \{0 \})(k) = k^2 \setminus \{ (0,0)\}$ such that $x := \pi(y_1, y_2)$ works. The map $\pi: k_v^2 \setminus \{0 \} = (\A^2_{k} \setminus \{0 \})(k_v) \to \Proj_{k}(n,m)(k_v)$ is surjective so we can choose non-empty opens $V_{1,v}, V_{2,v} \subset k^{\times}_v$ such that $V_{1,v} \times V_{2,v} \subset \pi^{-1}(U_v)$ for all $v \in S$. Now choose non-empty opens $V'_{1,v}, V'_{2,v} \subset k^{\times}_v$ and $W_{1,v}, W_{2,v} \subset k^{\times}_v$ open neighbourhoods of $1$ such that $(W_{i,v})^{m} \cdot V'_{i,v} \subset V_{i,v}$ for $i = 1,2$. These exist since multiplication is continuous.

Fix a place $w_0 \not \in S$. Apply strong approximation for $\A^1$ twice to choose $y'_1, y'_2 \in k^{\times}$ such that: 
\begin{itemize}
    \item $y'_i \in V'_{i,v}$ for $v \in S$ and $i = 1,2$
    \item $y'_1, y'_2 \in \mathcal{O}_v$ and at least one of $y'_1, y'_2$ lies in $\mathcal{O}^{\times}_v$ for $v \not \in S \cup \{ w_0 \}$.
\end{itemize}

Let $T$ be the finite set of places $v \not \in S \cup \{ w_0 \}$ such that $y'_1 y'_2 \not \in \mathcal{O}_{v}^{\times}$. Fix another place $w_1 \not \in S \cup T \cup \{w_0\}$. Choose a $z_1 \in k^{\times}$ with strong approximation such that:
\begin{itemize}
    \item $z_1 \in (W_{1,v})^{m}$ for $v \in S$
    \item $z_1 \in \mathcal{O}^{\times}_v$ for $v \in T$
    \item $n \mid w_0(z_1 y'_1)$ and $ m w_0(z_1 y'_1) < n w_0( y'_2)$
    \item $z_1 \in \mathcal{O}_v$ for $v \not \in S \cup T \cup \{w_0, w_1\}$.
\end{itemize}

Let $T'$ be the finite set of places $v \not \in S \cup \{w_0, w_1 \}$ such that $z_1 y'_1 y'_2 \not \in \mathcal{O}_v^{\times}$. Apply weak approximation to find a $z_2 \in k^{\times}$ such that:
\begin{itemize}
    \item $z_2 \in W_{2,v}$ for $v \in S$
    \item $z_2 \in \mathcal{O}^{\times}_v$ for $v \in T' \cup \{ w_0 \}$
    \item $n w_1(z_2) < \min(w_1(z_1), 0)$.
\end{itemize}

We now claim that the pair $y_1 := z_1 y'_1, y_2 := z_2^m y'_2$ works:
\begin{itemize}
    \item For $v \in S$ we have $y_i \in (W_{i,v})^{m} \cdot V'_{i,v} \subset V_{i,v}$ by construction, so $\pi(y_1, y_2) \in U_v$.
    \item We have $(y_1, y_2) \in (\A^2_{\mathcal{O}_v} \setminus \{0\})(\mathcal{O}_v)$ for $v \in T'$.
    \item If $\pi_{w_0}$ is a uniformizer in $\mathcal{O}_{w_0}$ then $(\pi_{w_0}^{-w_0(z_1 y'_1)} y_1, \pi_{w_0}^{-m w_0(z_1 y'_1) /n} y_2)\in \mathcal{O}_{w_0}^{\times} \times \mathcal{O}_{w_0}$. Hence $(y_1, y_2) \in \G_m(k_{w_0}) \cdot (\A^2_{\mathcal{O}_{w_0}} \setminus \{0\})(\mathcal{O}_{w_0})$ which maps into $\Proj_{\mathcal{O}_{w_0}}(\mathcal{O}_{w_0})$.
    \item Analogously we have $(z_2^{-n} y_1 , z_2^{-m} y_2) \in \mathcal{O}_v \times \mathcal{O}_v^{\times}$ if $v(z_2) < 0$ for $v \not \in S \cup T' \cup \{ w_0 \}$. Otherwise we have $(y_1, y_2) \in \mathcal{O}_v^\times\times\mathcal{O}_v$.
\end{itemize}
\end{proof}
\begin{remark}
The projection map $\pi: \A^2_{k} \setminus \{0 \} \to \Proj_{k}(n,m)$ played an important role in the proof. This is a special case of a more general construction, it is a \emph{universal torsor} \cite[(2.0.4)]{Colliot1987Descente} of $\Proj_{k}(n,m)$. The facts that  $\Pic(\A^2_{k_s} \setminus \{0 \}) = 0$ and $k_s[ \A^2_k \setminus \{0 \}]^{\times} = k_s^{\times}$ imply by \cite[Prop. 2.1.1]{Colliot1987Descente} that it is a universal torsor. 
\end{remark}

All tame stacky curves of genus $g < \frac{1}{2}$ are geometrically simply connected so this is an improvement over \cite[Thm. 1.0.1]{Christensen2020Topology} in two ways: it holds for more stacky curves and it gives strong approximation for all places and not only away from a single place.
\section{Cohomological obstructions for stacks}
\begin{definition}
Let $X$ be a Deligne-Mumford stack. Its \emph{Brauer group} is defined as
\begin{equation*}
    \Br X := H^2(X, \G_m).
\end{equation*}
\end{definition}

For the rest of this section let $k$ be a global field unless otherwise mentioned. Let $\Omega_k$ be its set of places. We recall from local class field theory that for each place $v$ of $k$ there exists an injection $\text{inv}_v: \Br k_v \to \Q/\Z$ which is surjective if $v$ is finite.

\begin{definition}
Let $X$ be a separated finite type Deligne-Mumford stack over $k$. The \emph{Brauer-Manin pairing} is defined as the map
\begin{equation*}
    X(\A_k) \times \Br X \to \Q/\Z: ((x_v)_{v \in \Omega_k}, A) \to \sum_{v \in \Omega_k} \text{inv}_v(A(x_v)).
\end{equation*}

The \emph{Brauer-Manin set} $X(\A_k)^{\Br}$ is the left kernel of this pairing.

The image $\Br_0(X) := \text{im}(\Br k \to \Br X)$ lies in the right kernel of this pairing due to \cite[Thm. 12.1.8]{Colliot2021Brauer} so we get an induced pairing
\begin{equation*}
    X(\A_k) \times \Br X/\Br_0 X \to \Q/\Z: ((x_v)_{v \in \Omega_k}, A) \to \sum_{v \in \Omega_k} \text{inv}_v(A(x_v)).
\end{equation*}
which we also call the Brauer-Manin pairing.

If $B$ is a subgroup of $\Br X$ or $\Br X/ \Br_0 X$ then we let $X(\A_k)^{B} \subset X(\A_k)$ be the set of adelic points which are orthogonal to $B$ for the Brauer-Manin pairing. Important subquotients are the \emph{algebraic Brauer group} $\Br_1 X := \ker(\Br X \to \Br X_{k_s})$ and its quotient $\Br_a X := \Br_1 X/ \Br_0 X$.

If $\pi: \mathcal{X} \to U$ is an integral model over a Dedekind scheme $U$ then we define the \emph{integral Brauer-Manin set} of $\mathcal{X}$ as $\mathcal{X}(\A_{U})^{\Br} := \mathcal{X}(\A_{U}) \cap X(\A_k)^{\Br}.$ If $\mathcal{X}$ is regular and integral then $\Br \mathcal{X}$ is a subgroup of $\Br X$ by \cite[Prop. 2]{Hassett2016Stable}. We then similarly define $\Br_0 \mathcal{X} := \text{im}( \Br U \to  \Br \mathcal{X})$, $\Br_1 \mathcal{X} := \ker(\Br \mathcal{X} \to H^0(U, R^2 \pi_* \G_m))$ and $\Br_a \mathcal{X} = \Br_1 \mathcal{X}/ \Br_0 \mathcal{X}$.

Let $S$ be the set of places of $k$ which do not come from a point in $U$, 
$S$ must in particular contain all archimedean places. We define the subgroup of \emph{locally constant algebras} as $\Beh(\mathcal{X}) := \ker( \Br_a \mathcal{X} \to \prod_{v \in S} \Br_a X_{k_v})$.
\end{definition}
\begin{proposition}
\hfill
\begin{enumerate}
    \item The Brauer-Manin pairing is well-defined.
    \item The image of $X(k)$ lies in $X(\A_k)^{\Br}$.
    \item For all $A \in \Br X$ and every place $v$ the evaluation map $X(k_v) \to \Q/\Z: x_v \to \operatorname{inv}_v(A(x_v))$ is locally constant.
    \item The Brauer-Manin set is closed.
\end{enumerate}
\end{proposition}
\begin{proof}
Everything but the third statement follows by changing the word scheme by Deligne-Mumford stack in the proofs of \cite[Prop. 12.3.1, Thm. 12.3.2]{Colliot2021Brauer}.

For the third statement, let $x_v \in X(k_v)$. There exists a smooth cover $U \to X$ such that $x_v$ lies in the image of $U(k_v)$ by \cite[Thm. 7.0.7]{Christensen2020Topology}. The statement then follows from the analogous statement for schemes, functoriality of the evaluation map and the fact that the morphism of topological spaces $U(k_v) \to X(k_v)$ is open.
\end{proof}

Recall that the twists of a $G$-torsor $\pi: Y \to X$ allow us to define the \emph{descent obstruction} \cite[\S 5.3]{Skorobogatov2001Torsors}, another obstruction to strong approximation. We define the obstruction set coming from $\pi$ as
\begin{equation*}
    X(\A_k)^{\pi} := \bigcup_{a \in H^1(k, G)}\pi_a(Y_a(\A_k)).
\end{equation*}

It is a classical result of Colliot-Th\'el\`ene and Sansuc \cite{Colliot1987Descente} that the Brauer-Manin obstruction is related to the descent obstruction for groups of multiplicative type. We will need a form of this relation which holds for stacks and in positive characteristic. The following proof is analogous to the proof of \cite[Thm.~4.2]{Harari2013DescentBrauer}.
\begin{proposition}
Let $X$ be a separated finite type Deligne-Mumford stack over $k$ and $G$ a smooth group scheme of multiplicative type over $k$. Let $\pi:Y \to X$ be a $G$-torsor. Then we have an inclusion
\begin{equation*}
    X(\A_k)^{\Br_1 X} \subseteq X(\A_k)^{\pi} = \bigcup_{a \in H^1(k, G)}\pi_a(Y_a(\A_k)).
\end{equation*}
\label{Brauer-Manin obstruction is contained in descent obstruction}
\end{proposition}
\begin{proof}
Let $G' = \Hom(G, \G_m)$ be the group of characters of $G$. It is a $\Gal(k_s/k)$-module which is finitely generated as an abelian group. Let $[Y] \in H^1(X,G)$ be the class corresponding to the torsor $Y \to X$. The cup product coming from the pairing $G' \times G \to \G_m$ defines a morphism 
\begin{equation*}
    H^{1}(k, G') \to H^2(X, \G_m)  = \Br X: \varphi \to \varphi \cup [Y]
\end{equation*}
whose image lands in $\Br_1 X$ because $\varphi = 0$ in $H^1(k_s, G') = 0$.

Let $(x_v)_v \in X(\A_k)^{\Br_1 X}$. Since $X \cong [Y/G]$ as a stack there exists for each place $v$ a class $a_v\in H^1(k_v, G)$ and a $y_v \in  Y_{a_v}(k_v)$ such that $x_v = \pi_{a_v}(y_v)$.

Recall that $[Y_{a_v}] = [Y] - a_v$ so we compute that for an arbitrary $\varphi \in H^1(k, G')$
\begin{equation}
    (\varphi \cup [Y])(x_v) = \pi_{a_v}^* (\varphi \cup [Y])(y_v) = \pi_{a_v}^* (\varphi \cup ([Y_{a_v}] + a_v))(y_v) = (\varphi \cup a_v) \in \Br k_v.
    \label{Computation of evaluation for Brauer group element coming from torsors}
\end{equation}

The group $G'$ has no $p$-torsion because $G$ is smooth \cite[Rem. 12.5]{Milne2017Algebraic}. The following sequence is thus exact by Poitou-Tate duality \cite[Thm. I.4.20]{Milne2006Duality}.
\begin{equation*}
    H^1(k, G) \to \sideset{}{'}\prod H^1(k_v, G) \xrightarrow{(a_v)_v \to (\varphi \to \sum_v \text{inv}_v(\varphi \cup a_v))}  H^1(k, G')^{\vee}.
\end{equation*}

The assumption $(x_v)_v \in X(\A_k)^{\Br_1 X}$ and \eqref{Computation of evaluation for Brauer group element coming from torsors} imply for all $\varphi \in H^1(k, G')$ that 
\begin{equation*}
    \sum_v \text{inv}_v(\varphi \cup a_v) = \sum_v \text{inv}_v((\varphi \cup [Y])(x_v)) = 0.
\end{equation*} There thus exists an $a \in H^1(k, G)$ such that the image of $a$ in $H^1(k_v, G)$ is $a_v$ for all $v$. Then $(x_v)_v = \pi_a((y_v)_v)$ as desired.
\end{proof}

The following theorem is a generalization of \cite[Thm. 4.1]{Harari2013DescentOpen} to stacks and  positive characteristic.
\begin{theorem}
Let $X$ be a separated finite type Deligne-Mumford stack over a global field $k$ and let $G$ be a smooth group scheme. Let $Y$ be a $G$-torsor over $X$. If the twist $Y_a$ satisfies strong approximation for all $a \in H^1(k,G)$ then descent along $\pi:Y \to X$ is the only obstruction to strong approximation for $X$.

If $G$ is of multiplicative type then the algebraic Brauer-Manin obstruction is the only obstruction to strong approximation for $X$.
\label{Strong approximation for quotients}
\end{theorem}
\begin{proof}
We have by definition that 
\begin{equation*}
    X(\A_k)^{\pi} = \bigcup_{a \in H^1(k, G)}\pi_a(Y_a(\A_k)).
\end{equation*}
But we know by Lemma \ref{Points on the base and on the torsor} that
\begin{equation*}
    X(k) = \coprod_{a \in H^1(k, G)}\pi_a(Y_a(k)).
\end{equation*}
Since the $\pi_a$ are continuous in the adelic topology and the image $Y_a(k)$ is dense in $Y_a(\A_k)$ we conclude that the image of $X(k)$ is dense in $X(\A_k)^{\pi}$.

If $G$ has multiplicative type then $X(\A_k)^{\Br_1 X} \subseteq X(\A_k)^{\pi}$ by Proposition \ref{Brauer-Manin obstruction is contained in descent obstruction}. The image of $X(k)$ is dense in $X(\A_k)^{\pi}$ so it is also dense in the subset $X(\A_k)^{\Br_1 X}$.
\end{proof}
\begin{theorem}
    If $X$ is a tame stacky curve of genus $g < 1$ over $k$ and $X(\A_k) \neq \emptyset$ then there exists a universal cover $T \to X$. Descent along $T$ is the only obstruction to strong approximation.

If the group $\pi_1(X_{k_s})$ is abelian then the algebraic Brauer-Manin obstruction is the only obstruction to strong approximation.
\end{theorem}
\begin{proof}
If $X(\A_k) \neq \emptyset$ then $X_{\text{coarse}}(\A_k) \neq \emptyset$. The coarse moduli space $X_{\text{coarse}}$ has genus $0$ so $X_{\text{coarse}} \cong \Proj^1$ and thus has a rational point lying on the dense open $U$ on which $X \to X_{\text{coarse}}$ is an isomorphism. So $X(k) \neq \emptyset$, fix $x \in X(k)$. 

The group $\pi_1(X_{k_s})$ is finite and has order coprime to $p$ because of \cite[Prop.~5.5]{Behrend2006Uniformization} so there exists a finite Galois cover $\pi: T_{k_s} \to X_{k_s}$ with Galois group $\pi_1(X_{k_s},x)$ where $T_{k_s}$ is a simply connected stacky curve. Fix $t \in T_{k_s}(k_s)$ mapping to $x$. The tuple $(T_{k_s} \to X_{k_s}, \pi_1(X_{k_s},x) ,t)$ is unique up to unique isomorphism and can thus be descended to a tuple $(T \to X, \pi_1(X_{k_s},x), t)$ where $T \to X$ is a morphism of stacky curves over $k$ with $T_{k_s}$ simply connected, $\pi_1(X_{k_s},x)$ is equipped with a potentially non-trivial Galois action and $t \in T(k)$ is a rational point.

The universal cover $T$ and all of its twists are simply connected stacky curves so combining Theorem \ref{Simply connected curves satisfy strong approximation} and Theorem \ref{Strong approximation for quotients} gives the desired statement.

If $\pi_1(X_{k_s})$ is abelian then $T \to X$ is a torsor under a finite \'etale group of multiplicative type so we can apply the second part of Theorem \ref{Strong approximation for quotients}.
\end{proof}
We have thus proven Theorem \ref{Brauer-Manin obstruction is only obstruction to strong approximation}. An interesting corollary is:
\begin{corollary}
If $X$ is a tame stacky curve of genus $g < \frac{5}{6}$ then the only obstruction to strong approximation is the Brauer-Manin obstruction.
\end{corollary}
\begin{proof}
The geometric fundamental group of all tame stacky curves of genus $g < \frac{5}{6}$ is abelian by \cite[Prop. 5.5]{Behrend2006Uniformization}.
\end{proof}
\begin{remark}
It will follow from the next section that if $X$ is a stacky curve of signature $(0 ; 2,3,5)$ then $\Br_1 X = \Br k$ and that if $k$ is a number field then $\Br X = \Br_1 X$. There is thus no Brauer-Manin obstruction to strong approximation. But $\pi_1(X_{k_s})$ is isomorphic to the icosahedral group $A_5$ by \cite[Prop. 5.5]{Behrend2006Uniformization}. The stack $X$ thus has a non-trivial \'etale cover which implies that it does not satisfy strong approximation by \cite[Thm. 8.4.10]{Poonen2017Rational}\footnote{The theorem is only proven for varieties, but the same argument also works for stacks.}.
\end{remark}
\section{Brauer groups of stacky curves}
\begin{proposition}\label{prop:Brauer_group_alg_closed_field}
Let $X$ be a stacky curve over an algebraically closed field $k$. Then $\Br X = 0$.
\end{proposition}
\begin{proof}
The Brauer group injects into $\Br k(X)$ by \cite[Prop. 2]{Hassett2016Stable}, which is trivial by Tsen's theorem \cite[Thm. 6.2.8]{Gille2017Central}.
\end{proof}

This implies that $\Br_1 X = \Br X$ over a perfect field $k$. The Hochschild-Serre spectral sequence implies that $\Br_1 X/\Br_0 X$ injects into $H^1(k, \Pic X_{k_s})$ \cite[Prop. 4.3.2]{Colliot2021Brauer}. In this section we will compute $H^1(k, \Pic X_{k_s})$ for a tame stacky curve with coarse moduli space $\Proj^1_k$.

The Picard group $\Pic X$ of a stacky curve $X$ over a field $k$ can be described as a class group of divisors \cite[Lemma 5.4.5]{Voight2022Canonical}. Let $\mathcal{P} \subset X_{\text{coarse}}$ be the stacky locus. For any point $P \in \mathcal{P}$ let $e_P$ be the order of the stabilizer of the point in $X$ lying above it and denote by $[P] \in \Pic X$ the class of the unique integral substack in $X$ lying above it. 

There exists a degree map $\Pic X \xrightarrow{\deg} \Q$, which can in general take non-integral values unlike the case of non-stacky curves. We have $\deg [P] = \frac{[k(P): k]}{e_P}$. Let $d_P$ be the denominator of this fraction in lowest terms, i.e. $d_P := e_P/\gcd([k(P): k], e_P)$. Let $d_X := \text{lcm}_{P \in \mathcal{P}}(d_P)$.

We define $\Pic^0 X$ as the kernel of the degree map $\Pic X \xrightarrow{\deg} \Q$. We then have
\begin{lemma}
Let $X$ be a tame stacky curve with $X_{\emph{coarse}} \cong \Proj^1_k$. The image of $\deg$ is $\frac{1}{d_X} \Z$ and we have an exact sequence
\begin{equation*}
     0 \to \Pic^0 X \to \bigoplus_{P \in \mathcal{P}} \Z/e_P \Z[P] \xrightarrow{[P] \to \deg [P]}  \left(\frac{1}{d_X}\Z\right) / \Z \to 0
\end{equation*}
\label{Picard group of stacky curve with coarse genus 0}
\end{lemma}
\begin{proof}
We use the description of $\Pic X$ as a divisor class group \cite[Lemma 5.4.5]{Voight2022Canonical}.

A stacky curve contains a dense open subscheme so $k(X) = k(X_{\text{coarse}}) = k(\frac{x}{y})$, where $x,y$ are coordinates on $X_{\text{coarse}}$. For any integral closed substack $Z \subsetneq X$ its image $Z_{\text{coarse}} \subset X_{\text{coarse}}$ is defined by an equation $f(x,y)$ of degree $[k(Z_{\text{coarse}}): k]$.  Let $G_Z$ be the stabilizer group of $Z$. The stacky curve $X$ is a root stack over $X_{\text{coarse}}$ so we can see from the local description of root stacks \cite[Ex. 2.4.1]{Cadman2007Tangency} that the divisor $\text{div}(f)$ on $X$ corresponding to $f$ is equal to $\# G_Z[Z]$. The fact that $f(x,y)/ x^{[k(Z_{\text{coarse}}): k]} \in k(X)$ implies that $\# G_Z [Z] = [k(Z_{\text{coarse}}): k] \text{div}(x) \in \Pic X$. 

These relations generate all relations coming from $k(X)$ because every element is a product of a constant and rational functions of the form $f(x,y)/x^{[k(Z_{\text{coarse}}): k]}$. It follows that $\Pic X$ is equal to the group generated by the divisors $\text{div}(x)$ and $[P]$ for all $P \in \mathcal{P}$ with respect to the relations $e_P [P] = [k(P): k] \text{div}(x)$. The lemma follows.
\end{proof}
\begin{remark}
Suppose that $k$ is separably closed. We then have $d_P = e_P$ for all $P \in \mathcal{P}$ so for a stacky curve $X$ with signature $(0; e_1, \cdots, e_r)$ we have $\Pic^0 X = 0$ if and only if all the $e_i$ are pairwise coprime. One can check that if $g < 1$ the only such curves are the simply connected ones and the stacky curve of signature $(0; 2,3,5)$.
\label{The stacky curves with Pic^0 trivial}
\end{remark}

\begin{lemma}
Let $X$ be a tame stacky curve with $X_{\emph{coarse}} \cong \Proj^1_k$. Let $U$ be a normal scheme with fraction field $k$ such that the action of $\Gal(k_s/k)$ on $\Pic X_{k_s}$ factors through $\pi_1(U)$, e.g. $U = \Spec k$. We can then identify $\Pic X_{k_s}$ and $\Pic^0 X_{k_s}$ with the corresponding locally constant \'etale sheaves on $U$. We then have an exact sequence
\begin{equation*}
    0 \to \Z/ \frac{d_{X_{k_s}}}{d_{X}} \Z \to H^1(U, \Pic^0 X_{k_s}) \to H^1(U, \Pic X_{k_s}) \to 0.
\end{equation*}
\label{H1 of the Kummer map is surjective}
The first map sends the generator $1$ to the cocycle
\begin{equation*}
    \varphi: \pi_1(U) \to \Pic^0 X_{k_s}: \sigma \to \sigma D - D
\end{equation*}
for any divisor $D$ on $X_{k_s}$ with degree $1/d_{X_{k_s}}$.
\label{Computation of H^1(U, Pic)}
\end{lemma}
\begin{proof}
Note that $H^1(U, \Z) = 0$ so the long exact sequence in cohomology coming from the degree exact sequence $0 \to \Pic^0 X_{k_s} \to \Pic X_{k_s} \to \frac{1}{d_{X_{k_s}}}\Z \to 0$ is
\begin{equation*}
     (\Pic X_{k_s})^{\pi_1(U)} \xrightarrow{\deg} \frac{1}{d_{X_{k_s}}}\Z \to H^1(U, \Pic^0 X_{k_s}) \to H^1(U, \Pic X_{k_s}) \to 0.
\end{equation*}
The map $\frac{1}{d_{X_{k_s}}}\Z \to H^1(U, \Pic^0 X_{k_s})$ sends the generator to $\varphi$ by the explicit description of this boundary map. The assumption $X_{\text{coarse}} \cong \Proj^1_k$ implies that $X(k) \neq \emptyset$ and thus that $(\Pic X_{k_s})^{\pi_1(U)} = \Pic X$, so the image of the degree map is $\frac{1}{d_X} \Z$.
\end{proof}
We can now construct a section of the map $\Br_1 X \to H^1(k, \Pic X_{k_s})$ using this description of the Picard group and work of Skorobogatov. We will utilise the notion of the \emph{type} \cite[Def. 2.3.1]{Skorobogatov2001Torsors} of a torsor.

Let $T \to X$ be a torsor with type $\Pic^0 X_{k_s} \to \Pic X_{k_s}$. It is a $\Hom(\Pic^0 X_{k_s}, \G_m)$-torsor. We have $X_{\text{coarse}} \cong \Proj^1_k$ and thus $X(k) \neq \emptyset$. Such a point defines a splitting of the exact sequence \cite[Cor. 2.3.9]{Skorobogatov2001Torsors} so such a torsor exists.
\begin{proposition}
Let $X$ be a tame stacky curve with $X_{\emph{coarse}} \cong \Proj^1_k$. The group $\Br_1 X$ is generated by $\Br k$ and the Brauer classes obtained by cupping elements of $H^1(k, \Pic^0 X_{k_s})$ with $[T]$.
\label{Brauer group is generated by cyclic algebras}
\end{proposition}
\begin{proof}
Follows from \cite[Thm. 4.1.1]{Skorobogatov2001Torsors}, Lemma \ref{Computation of H^1(U, Pic)} and the definition of $T$.
\end{proof}
\begin{example}
Let $n, m \neq 1$ and $X$ a tame stacky curve over $k$ with signature $(0; n, m)$. Assume that $X_{\text{coarse}} \cong \Proj^1_k$. Let $P, Q \in X_{\text{coarse}}(k_s)$ be the stacky points which have stabilizers of order $n, m$ respectively and define $r := \gcd(n,m)$. We see from Lemma \ref{Picard group of stacky curve with coarse genus 0} that $\Pic^0 X_{k_s} \cong \Z / r\Z$ as an abstract group. The generator is given by $\frac{n}{r}[P'] - \frac{m}{r}[Q']$. Here $P', Q' \in X_{k_s}$ are the unique integral closed substacks lying over $P, Q$ respectively.

If the action of $\Gal(k_s/k)$ on the stacky points is trivial then we may assume that $P = [1: 0], Q = [0:1]$. In this case $\Gal(k_s/k)$ acts trivially on $\Pic^0 X_{k_s} \cong \Z/ r\Z$ so $T \to X$ is a $\mu_r$-torsor. One can check via an explicit computation that the $\mu_r$-torsor $T := \Proj^1(\frac{n}{r}, \frac{m}{r}) \to X$ which on coarse moduli spaces is given by $[s:t] \to [s^r: t^r]$ has type $\Pic^0 X_{k_s} \to  \Pic X_{k_s}$.

If the action on the stacky points is non-trivial then $n = m$ and the action becomes trivial after a quadratic extension $K/k$. One can then recover a choice of $[T]$ by descending the above construction. To make this explicit, let $\sigma \in \Gal(K/k)$ be the generator and $f(X,Y) = (X - \alpha Y)(X - \sigma(\alpha) Y) \in k[X,Y]$ a homogeneous polynomial of degree 2 defining the stacky locus. The relevant torsor over $K$ is given by the map $\Proj^1_{K; s,t} \to X_K$ which on the coarse moduli spaces is given by $[s:t] \to [\alpha s^r - \sigma(\alpha) t^r: s^r - t^r]$. The action on $\Proj^1_K$ which we will use to descend is the map $[s:t] \to [\sigma(t): \sigma(s)]$. The result of this descent is the curve $T := \{ N_{K/k}(U - \alpha V) = W^2\} \subset \Proj^2$. The map $T \to X$ is given by $[u: v: w] \to [ \alpha (u - \sigma(\alpha) v)^r - \sigma(\alpha) w^{r}: (u - \sigma(\alpha) v)^r - w^r]$.

Let us assume that $n = m = r = 2$. We can use $w^2 = (u - \alpha v)(u - \sigma(\alpha)v)$ to simplify the image of $[u:v: w]$ to
\begin{equation*}
    [(u - \sigma(\alpha) v)(\alpha u - \alpha \sigma(\alpha) v - \sigma(\alpha) u + \sigma(\alpha) \alpha v): (u - \sigma(\alpha) v)(u - \sigma(\alpha) v -u + \alpha v)] = [u: v].
\end{equation*} This is the same torsor which shows up in the example of Bhargava-Poonen \cite{Bhargava2020Stacky}.
\label{The universal cover there are at most two stacky points}
\end{example}

One can do similar constructions for other stacky curves $X$ with $X_{\text{coarse}} \cong \Proj^1$.

\section{The elementary obstruction}
Given a Deligne-Mumford stack $X$ we let $\mathcal{D}(X)$ be the derived category of sheaves on the (small) \'etale site of $X$. Any morphism $f: Y \to X$ of Deligne-Mumford stacks defines a right derived functor $R f_*: \mathcal{D}(Y) \to \mathcal{D}(X)$ and a pullback $f^* :\mathcal{D}(X) \to \mathcal{D}(Y)$. The global sections functor $\Gamma$ has a right derived functor $R\Gamma(X, \cdot): \mathcal{D}(X) \to \mathcal{D}(\text{Ab})$ where $\mathcal{D}(\text{Ab})$ is the derived category of abelian groups. We denote by $H^p(X, \cdot) := H^p(R \Gamma(X, \cdot))$ the hypercohomology. We can extend the constructions of \cite[\S 1]{Harari2013DescentOpen} to stacks.
\begin{construction}
Let $\pi:X \to S$ be a morphism of Deligne-Mumford stacks. We denote $KD(X) := (\tau_{\leq 1} R \pi_* \G_m)[1]$ the shift of the truncation of $R \pi_* \G_m$. It is concentrated in degrees $0, -1$ and $H^{-1}(KD(X)) = \pi_* \G_m, H^0(KD(X)) = R^1 \pi_* \G_m$.

The map $\G_m \to R \pi_* \G_m$ defines a morphism $\G_m[1] \to KD(X)$. Its cone will be denoted $KD'(X)$. We thus have a distinguished triangle in $\mathcal{D}(S)$
\begin{equation}
    \G_m[1] \to KD(X) \to KD'(X) \xrightarrow{w_X} \G_m[2].
    \label{The fundamental triangle of KD}
\end{equation}
\end{construction}
If $\pi$ has a section $s: S \to X$ then the map $R \pi_* \G_m \to R \pi_*  R s_* \G_m \cong \G_m$ induces a retract $KD(X) \to \G_m[1]$ of $\G_m[1] \to KD(X)$. This implies that the triangle \eqref{The fundamental triangle of KD} splits, equivalently that $w_X = 0$. We call $w_X \in \Hom_S( KD'(X), \G_m[2]) = \text{Ext}_S^2(KD'(X), \G_m)$ the \emph{elementary obstruction}. If $S$ is the spectrum of a field and $\pi_*\G_m = \G_m$ then $w_X$ is the elementary obstruction occurring in \cite[Prop 2.2.4]{Colliot1987Descente}.

A more explicit construction of the complex $KD(X)$ when $X$ is a variety smooth over $S$ is given in \cite[App. A]{Harari2013DescentOpen}.


The complexes $KD(X)$ and $KD'(X)$ are related to the Brauer group as follows.
\begin{lemma}
\hfill
\begin{enumerate}
    \item $H^1(S, KD( X)) \cong \Br_1 X := \ker(\Br X \to  H^0(S, R^2 \pi_* \G_m))$.
    \item $\ker(H^1(S, KD'( X )) \xrightarrow{w_{X *}} H^3(S, \G_m)) \cong  \Br_a X := \Br_1 X / \emph{im}(\Br S \to \Br X)$.
\end{enumerate}
\label{Cohomology of KD and the Brauer group}
\end{lemma}
\begin{proof}
This follows from a standard computation with distinguished triangles, see e.g. \cite[Lemma 2.1]{Harari2008Local-global} and \cite[p. 7]{Harari2013DescentOpen}.
\end{proof}

If $T$ is a smooth group of multiplicative type over $S$ with sheaf of characters $\hat{T}$ then we have the following exact sequence as in \cite[Prop. 1.1]{Harari2013DescentOpen}.
\begin{equation}
    H^1(S, T) \to H^1(X, T) \xrightarrow{\chi} \Hom_S(\hat{T}, KD'(X)) \xrightarrow{\partial} H^2(S, T) \to H^2(X, T).
    \label{Long exact sequence for tori coming from the triangle}
\end{equation}

We can compute the complex $KD'(\mathcal{X})$ for stacky curves.
\begin{lemma}
Let $S$ be a regular integral scheme. Let $\mathcal{X} \to S$ be a tame relative stacky curve whose generic fiber $X$ has a coarse moduli space of genus $0$. The complex $KD'(\mathcal{X})$ is concentrated in degree $0$ and is \'etale locally a constant finitely generated group, i.e. $\pi_* \G_m = \G_m$ and $R^1 \pi_* \G_m$ is \'etale locally a constant finitely generated group. Moreover, the order of the torsion part of $R^1 \pi_* \G_m$ is invertible on $S$.
\label{R pi G_m of a stacky curve}
\end{lemma}
\begin{proof}
We have $\pi_* \G_m = \G_m$ because the map is proper and all fibers are stacky curves and thus geometrically connected.

To show that $R^1 \pi_* \G_m$ is locally constant we apply \cite[\href{https://stacks.math.columbia.edu/tag/0GKC}{Tag 0GKC}]{stacks-project}. It remains to show that if $S$ is the spectrum of a regular strictly henselian ring $R$ with fraction field $K$ and residue field $\kappa$ then the specialization map $\Pic \mathcal{X} = (R^1 \pi_* \G_m)(S) \to (R^1 \pi_* \G_m)(K_s) = \Pic X_{K_s}$ is an isomorphism.

We note that by Lemma \ref{Picard group of stacky curve with coarse genus 0} $\Pic X_{K_s}$ is finitely generated and the order of its torsion subgroup divides the product of the size of the stabilizer groups. Proposition \ref{Prop:Global structure stacky curves} implies that if $\mathcal{X}$ is tame then all of these cardinalities are invertible on $S$ and thus also on $\mathcal{X}$.

We apply \cite[Prop. 2.3.2]{Fringuelli2022Picard} to see that $\Pic \mathcal{X} = \Pic X$ since $\mathcal{X}$ is regular and $\Pic S = 0$. It remains to show that $\Pic X = \Pic X_{K_s}$, which by Lemma \ref{Picard group of stacky curve with coarse genus 0} is equivalent to $\mathcal{P}(K) = \mathcal{P}(K_s)$ for the stacky locus $\mathcal{P} \subset X_{\text{coarse}, K}$.

The closure in $\mathcal{X}_{\text{coarse}}$ of $\mathcal{P}$ is finite \'etale over $S = \Spec R$ by Proposition \ref{Prop:Global structure stacky curves}. But $R$ is strictly henselian so the closure must be a disjoint union of copies of $S$ which implies that $\mathcal{P}(K) = \mathcal{P}(K_s)$.
\end{proof}
\subsection{The elementary obstruction for stacky curves}\label{sec:elementary_obstruction}
We will now analyze the elementary obstruction for stacky curves. For the rest of this subsection we will look at the following situation.

Let $U$ be a finite type Dedekind scheme whose field of fractions $k$ is a global field. Let $S$ be the finite set of places of $k$ which do not define points of $U$, it must contain all the archimedean places. Let $\pi: \mathcal{X} \to U$ be a relative tame stacky curve whose generic fiber $X := \mathcal{X}_k$ has genus $g < 1$. Let $\mathcal{P} \subset X_{\text{coarse}}$ be the stacky locus. In this subsection we will prove the first and third parts of Theorem \ref{The elementary obstruction is the only obstruction}.

By Lemma \ref{R pi G_m of a stacky curve} $KD'(\mathcal{X}) = R^1 \pi_* \G_m$ is locally constant. The stalk at the generic point is $(R^1 \pi_* \G_m)_{k_s} = \Pic X_{k_s}$ so $R^1 \pi_* \G_m$ can be identified with the $\pi_1(U)$-module $\Pic X_{k_s}$. Let $\hat{T}^0 \subset R^1 \pi_* \G_m$ be the subsheaf corresponding to $\Pic^0 X_{k_s}$ and $T^0 := \Hom(\hat{T}^0, \G_m)$, $T := \Hom(R^1 \pi_* \G_m, \G_m)$ the dual groups of multiplicative type. Let $T_{\deg} := \ker(T \to T^0)$, it has group of characters $\Pic X_{k_s}/\Pic^0 X_{k_s} \cong \Z$ so $T_{\deg} \cong \G_m$.

We can now prove Theorem \ref{The elementary obstruction is the only obstruction}.
\begin{theorem}
Let $\mathcal{X} \to U$ be a relative tame stacky curve whose generic fiber $X$ has genus $g < 1$. The elementary obstruction is then the only obstruction to the existence of an integral point on $\mathcal{X}$, i.e. if $w_{\mathcal{X}} = 0$ then $\mathcal{X}(U) \neq \emptyset$.
\label{The elementary obstruction is the only obstruction if genus < 1}

Moreover, if the signature of $X$ is not equal to $(0; n,m)$ then $\mathcal{X}(U) \neq \emptyset$.
\end{theorem}
\begin{proof}
The second statement is Lemma \ref{lem:Stacky curves with non-abelian fundamental group have integral points} and Proposition \ref{prop:signature_(2,2,2)}. It remains to deal with the case when $X$ has signature $(0; n,m)$ where $n,m$ can be equal to $1$.

The triangle \eqref{The fundamental triangle of KD} splits since $w_{\mathcal{X}} = 0$. The long exact sequence \eqref{Long exact sequence for tori coming from the triangle} coming from this triangle splits so there exists a $T$-torsor $\mathcal{V} \to \mathcal{X}$ such that $\chi([\mathcal{V}])$ is the identity $R^1 \pi_* \G_m \to R^1 \pi_* \G_m$. 

Let $\mathcal{Y} := \mathcal{V}/T_{\deg}$. Then $\mathcal{Y} \to \mathcal{X}$ is a $T^0 = T/T_{\deg}$-torsor such that $\chi([\mathcal{Y}])$ is the inclusion morphism $\hat{T}^0 \subset R^1 \pi_* \G_m$ by functoriality. It follows from \cite[(2.6), Lemma 2.3.1:(iii)]{Skorobogatov2001Torsors} and functoriality that all geometric fibers of $\mathcal{Y} \to U$ are connected. The stack $\mathcal{Y}$ is smooth and proper and the order of its stabilizer groups has to divide $nm$, we conclude that it is a relative tame stacky curve.

The map $\mathcal{Y} \to \mathcal{X}$ has degree $\ell := \# \Pic^0 X_{k_s} = \gcd(n,m)$ by Lemma \ref{Picard group of stacky curve with coarse genus 0}. The Riemann-Hurwitz formula and inspection of the list \cite[Prop.~5.5]{Behrend2006Uniformization} show that the only possibility for the signature of $\mathcal{Y}$ is $(0; \frac{n}{\ell}, \frac{m}{\ell})$. Let $Y$ be the generic fiber of $\mathcal{Y}$.

The map $\mathcal{V} \to \mathcal{Y}$ is the $\G_m \cong T_{\deg}$-torsor corresponding to the image of $\Pic X_{k_s} \to \Pic Y_{k_s} \cong \Z$. This map is surjective by degree considerations and the computations in Lemma \ref{Picard group of stacky curve with coarse genus 0}. The existence of this torsor implies that $\Pic Y_{k_s} = \Pic Y \cong \Z$, so $\Pic Y$ contains a degree $1$ divisor class. 

If $n = m = \ell$ then $Y$ is a genus $0$ curve and $\mathcal{Y}$ is a scheme. By the valuative criterion of properness we have $\mathcal{Y}(U) = Y(k)$. Note now that $Y \cong \Proj^1_k$ as $\Pic Y = \Pic Y_{k_s}$.

If $n \neq m$ then $\mathcal{Y}$ has to be a twist of $ \Proj(\frac{n}{\ell}, \frac{m}{\ell})$. If $n \mid m$ or $m \mid n$ then the group scheme of automorphisms $G := \text{Aut}(\Proj(\frac{n}{\ell}, \frac{m}{\ell}))$ is the subgroup of $PGL_2$ which fixes one point, i.e. the group $\G_a \rtimes \G_m$ of affine transformations of $\A^1$. Otherwise we have $n \neq \ell$ and $m \neq \ell$ and $G$ is the subgroup of $PGL_2$ which fixes two points, i.e. $\G_m$.

Let $R$ be a field or DVR, then $H^1(R, \G_m) = H^1(R, \G_a) = 0$. So in both cases we have $H^1(R, G) = 0$. As twists are classified by $H^1$ this implies $\mathcal{Y}$ becomes isomorphic to $\Proj(\frac{n}{\ell}, \frac{m}{\ell})$ after base change to $\mathcal{O}_v$ for $v \not \in S$ and after base change to $k_v$ for $v  \in S$, in particular $\mathcal{Y}(\A_U) \neq \emptyset$. 

Theorem \ref{Simply connected curves satisfy strong approximation} implies that $\mathcal{Y}(U) \neq \emptyset$ and thus that $\mathcal{X}(U) \neq \emptyset$.
\end{proof}
The following example shows that the assumption $w_{\mathcal{X}} = 0$ is necessary.
\begin{example}\label{ex:Root stacks on Hirzebruch surfaces}
    Let $n, m \in \N$ be such that $\ell := \gcd(n, m) > 1$ and $q$ be a prime power coprime to $n,m$. Let $U := \Proj^1_{\F_q}$, $F := \Proj^1_U(\mathcal{O}_U \oplus \mathcal{O}_U(1))$ be the Hirzebruch surface over $\F_q$ and $\pi:F \to U$ the structure map. Let $\mathcal{O}_F(1)$ be the tautological line bundle on this projective bundle.

    Let $D \subset F$ be the divisor defined by $1 \oplus 0 \in H^0(F, \mathcal{O}_{F}(1)) = H^0(U, \mathcal{O}_U \oplus \mathcal{O}_U(1))$ and $E \subset F$ the divisor defined by $0 \oplus 1 \in H^0(F, \mathcal{O}_{F}(1) \otimes \pi^*\mathcal{O}_U(-1)) = H^0(U, \mathcal{O}_U(-1) \oplus \mathcal{O}_U)$. 

    One checks that the maps $D \to U$ and $E \to U$ are isomorphisms and that $D \cap E = \emptyset$. It follows that $\mathcal{X} := F[\sqrt[n]{D}][\sqrt[m]{E}]$ is a tame relative stacky curve over $U$ of signature $(0 ; n, m)$. If $v$ is a place of $k := \F_q(t)$ then we have $F_{\mathcal{O}_v} \cong \Proj^1_{\mathcal{O}_v}$ and $\mathcal{X}_{\mathcal{O}_v} \cong \Proj^1_{\mathcal{O}_v}\big[\sqrt[n]{[0:1]}\big]\big[\sqrt[m]{[1:0]}\big]$. This has the integral point $[1:1]$ so $\mathcal{X}$ has integral points locally everywhere.

    We will show that $\mathcal{X}(U) = \emptyset$. Let $U \to \mathcal{X}$ be an integral point. By the definition of root stacks this corresponds to a section $s: U \to F$ for which $s^* \mathcal{O}_F(D)$, resp. $s^* \mathcal{O}_F(E)$, is an $n$-th, resp. $m$-th, power. They are thus both $\ell$-th powers. But $\mathcal{O}_F(D) \cong \mathcal{O}_{F}(1)$ and $\mathcal{O}_F(E) \cong \mathcal{O}_{F}(1) \otimes \pi^*\mathcal{O}_U(-1)$ by construction so $s^* \mathcal{O}_{F}(D) \otimes s^* \mathcal{O}_F(E)^{\vee} \cong s^* \pi^* \mathcal{O}_U(1) = \mathcal{O}_U(1)$ is not a $\ell$-th power which provides a contradiction.
\end{example}
\subsection{\texorpdfstring{The $(0; 2,2,2)$ case}{The (0;2,2,2)-case}}\label{sec:(0;2,2,2)_case}
In this subsection we will prove the following proposition.
\begin{proposition}\label{prop:signature_(2,2,2)}
    Let $U$ be a finite type Dedekind scheme whose fraction field $k$ is a global field and $\mathcal{X} \to U$ a relative tame stacky curve with signature equal to $(0; 2,2,2)$. Then $\mathcal{X}(U) \neq \emptyset$.
\end{proposition}
Note that by tameness we have that $2$ is invertible on $U$. We let $X$ be the generic fiber of $\mathcal{X}$. Let $\mathcal{P} \subset \mathcal{X}_{\mathrm{coarse}}$ be the stacky locus. This has degree $3$ over $U$ by assumption, so we have $X_{\mathrm{coarse}} \cong \Proj^1_{k}$ as it is a genus $0$ curve with a degree $3$ divisor.

We will now classify such curves. 
\begin{lemma}\label{lem:classification_(0;2,2,2)}
    Let $\varphi_{\mathcal{X}}: \Gal(k_s/k) \to S_3$ be the morphism given by the Galois action on the three points $\mathcal{P}(k_s)$, which is well-defined up to conjugacy. This map factors through $\pi_1(U)$ and the assignment $\mathcal{X} \to \varphi_{\mathcal{X}} \in H^1(U, S_3)$ induces a bijection between isomorphism classes of relative tame stacky curves $\mathcal{X} \to U$ of signature $(0; 2,2,2)$ and $H^1(U, S_3)$.
\end{lemma}
\begin{proof}
    The fact that $\operatorname{PGL}_2$ acts simply transitively on 3 ordered points of $\Proj^1$ implies that $\mathcal{X}$ is locally isomorphic to $\Proj^1_U[\sqrt[2]{[1:0]}][\sqrt[2]{[0:1]}][\sqrt[2]{[1:1]}]$ and that the automorphism group of this latter curve is $S_3$. The lemma then follows from descent and that the action of this $S_3$ on the stacky locus is the usual permutation action.
\end{proof}
Let $A := \Pic^0 X_{k_s}$. By Lemma \ref{Picard group of stacky curve with coarse genus 0} we have $A \cong (\Z/2\Z)^2$ as abstract groups and the action of $\Gal(k_s/k)$ is given by composing $\varphi_{\mathcal{X}}$ with $S_3 \cong \operatorname{GL}_2(\Z/2 \Z)$. In particular, the Galois action factors through $\pi_1(U)$ so $A$ can be identified with a finite \'etale group scheme over $U$. Note that $\hat{A} \cong A$ by a direct computation.
\begin{lemma}\label{lem:universal_cover_(0;2,2,2)}
    There exists an $A$-torsor $\pi: \mathcal{Y} \to \mathcal{X}$ over $U$ such that $\mathcal{Y}$ is a relative genus $0$ curve. Moreover, every pair consisting of a relative genus $0$ curve $\mathcal{Y}'$ and an embedding $A \subseteq \Aut_U(\mathcal{Y}')$ is isomorphic to a twist $\mathcal{Y}_a$ with its $A$-action for some $a \in H^1(U, A)$, up to an automorphism of $A$.
\end{lemma}
\begin{proof}
    Let $\mathcal{Y}^{\mathrm{split}} := \{Z_0^2 + Z_1^2 + Z_2^2 = 0 \} \subseteq \Proj^2_U$. This is a conic which has an action of $S_3$ given by permuting the $Z_0, Z_1, Z_2$ and a $(\Z/2 \Z)^2 \cong \mu_2^2$-action
    \[
        \mu_2^2 \times \mathcal{Y}^{\mathrm{split}} \to \mathcal{Y}^{\mathrm{split}}: (\zeta, \eta; Z_0, Z_1, Z_2) \to (\zeta Z_0, \eta Z_1, Z_2).
    \]
    Note that the $S_3$-action and the $(\Z/2\Z)^2$ action extend to an action of the semi-direct product $(\Z/2\Z)^2 \rtimes S_3$, with the action $S_3 \cong \operatorname{GL}_2(\Z/2 \Z)$, so the $S_3$ action extends to an action on the quotient $\mathcal{X}^{\mathrm{split}} := \mathcal{Y}^{\mathrm{split}}/(\Z/2 \Z)^2$.

    Each geometric fiber of $\mathcal{X}^{\mathrm{split}}$ is a stacky curve with fundamental group $(\Z/2\Z)^2$ (as the geometric fiber of $\mathcal{Y}^{\mathrm{split}}$ is its universal cover), so by \cite[Prop 5.6]{Behrend2006Uniformization} it is a stacky curve of signature $(0; 2,2,2)$. The stacky locus is given by the image of $\{Z_0 Z_1 Z_2 = 0\}$ so the $S_3$-action on it is the permutation action.

    We can then twist the $(\Z/2 \Z)^2$ torsor $\mathcal{Y}^{\mathrm{split}} \to \mathcal{X}^{\mathrm{split}}$ by $\varphi_{\mathcal{X}}: \pi_1(U) \to S_3$. The Galois action on the stacky locus is equal to $\varphi_{\mathcal{X}}$ by construction, so the twist of $\mathcal{X}^{\mathrm{split}}$ is $\mathcal{X}$ by Lemma \ref{lem:classification_(0;2,2,2)}. It is similarly simple to see that the $\pi_1(U)$-action on the twist of $(\Z/2 \Z)^2$ is the same as the action on $A$. This proves the first part of the lemma.

    Let now $\mathcal{Y}'$ be a relative genus $0$ curve and $A \subseteq \Aut_U(\mathcal{Y}')$ an embedding. The quotient $\mathcal{X}' := \mathcal{Y}'/A$ is a relative tame stacky curve of signature $(0; 2,2,2)$ by the same argument as for $\mathcal{X}^{\mathrm{split}}$.

    Let $X'$ be the generic fiber of $\mathcal{X}'$ and $\alpha \in \Hom(A, \Pic X_{k_s}), \alpha' \in \Hom(A, \Pic X'_{k_s})$ be the respective images of the torsors $\mathcal{Y} \to \mathcal{X}$, $\mathcal{Y}' \to \mathcal{X}'$ under the map $\chi$ of \eqref{Long exact sequence for tori coming from the triangle} (where we used Lemma \ref{R pi G_m of a stacky curve} to compute $KD'(\mathcal{X}), KD'(\mathcal{X}')$).

    As $A$ is torsion we have that $\alpha'$ factors through $\Pic^0 X'_{k_s}$. The fact that $\mathcal{Y}, \mathcal{Y}'$ have geometrically connected fibers and \cite[(2.6), Lemma 2.3.1:(iii)]{Skorobogatov2001Torsors} imply that $\alpha$ and $\alpha'$ are injective. It follows that $\alpha': A \to \Pic^0 X'_{k_s}$ is an isomorphism, so by Lemma \ref{lem:classification_(0;2,2,2)} there exists an isomorphism $\mathcal{X} \cong \mathcal{X}'$, which we fix.

    As $\alpha, \alpha': A \to \Pic^0 X_{k_s}$ are both isomorphisms it follows that we can make them equal by applying an isomorphism of $A$. The second part of the lemma then follows from the exactness of \eqref{Long exact sequence for tori coming from the triangle}.
\end{proof}

Let $\mathcal{Y} \to \mathcal{X}$ be as in the above lemma and let $Y$ be the generic fiber of $\mathcal{Y}$. Now fix an elliptic curve $E/k$ such that $E[2] \cong A$, for example one can take the elliptic curve $y^2 = f(x)$ where $f(x)$ is a cubic polynomial defining the stacky locus $\mathcal{P}(k_s) \subset \A^1_k \subset \Proj^1_k \cong X_{\mathrm{coarse}}$ after choosing a point at infinity disjoint to $\mathcal{P}(k_s)$.

Let $C := E/\{+1, -1\}$ be the quotient by the inverse action. This is a copy of $\Proj^1_k$. The group $\{+1, -1\}$ preserves $E[2]$ so the action of $E[2]$ on $E$ by translation descends to an action on $C$. Moreover, this action is faithful.

It follows from Lemma \ref{lem:universal_cover_(0;2,2,2)} applied to some sufficiently small dense open $U' \subset U$ that there exists $\beta \in H^1(U', A)$ such that $C_{\beta} \cong  Y$ as curves equipped with an action of $A$, potentially after changing the choice of isomorphism $A \cong E[2]$

For each place $v$ of $k$ let $q_v: H^1(k_v, A) \to \Br k_v[2] \subseteq \frac{1}{2}\Z/\Z$ be the quadratic form induced by the theta group of $E[2]$, as in \cite[\S4.2]{Poonen2012Random}. This map has the following geometric description. Given $a_v \in H^1(k_v, A)$ let $C_{a_v}$ be the twist of $C$ by $a_v$. The map $E \to C$ is induced by the line bundle $\mathcal{O}(2[0])$, where $[0] \in E$ is the unit. It follows from \cite[Prop.~2.1, 2.3]{ONeil2002Period} that $q_v(a_v) = [C_{a_v}] \in \Br k_v[2]$, where $[C_{a_v}] \in \Br k_v[2]$ is the class of $C_{a_v}$ as a Brauer-Severi variety..

Let $S$ be the finite set of places of $k$ which do not define a point of $U$. Recall \cite[p.~249]{Poonen2012Random} that a \emph{maximally isotropic subspace} $L \subset V$ of a quadratic space $(V, q)$ is a linear subspace such that $q(L) = 0$ and $L$ is equal to its orthogonal complement under the induced bilinear pairing.
\begin{lemma}\label{lem:existence_of_good_twist}
    There exists $a \in H^1(U, A)$ such that $q_v(a + \beta) = 0$ for all places $v \in S$.
\end{lemma}
\begin{proof}
    For each $v \in S$ consider $V_v := H^1(k_v, A)$ as a vector space over $\F_2$. The map $q_v: V_v \to \frac{1}{2} \Z/\Z \cong \F_2$ is then a quadratic form by \cite[Cor.~4.7]{Poonen2012Random}. Moreover, the corresponding bilinear form $\langle x,y \rangle_{q_v} := q_{v}(x + y) - q_v(x) - q_v(y)$ is equal to the local Tate pairing induced by the pairing $A \times A \to \G_m$ coming from the isomorphism $E[2] \cong A$ and the Weil pairing by loc.~cit. This pairing is nondegenerate by local Tate duality \cite[I.2.3, I.2.13(a), III.6.10]{Milne2006Duality}.

    The quadratic spaces $(V_v, q_v)$ have a maximal isotropic subspace by  \cite[Prop.~4.11]{Poonen2012Random}. Let $L \subseteq \bigoplus_{v \in S} V_v$ be the image of $H^1(U, A)$. This is a linear subspace equal to its orthogonal complement with respect to the bilinear form associated to $(\bigoplus_{v \in S} V_v, \sum q_v)$ by the same argument as in \cite[Thm.~4.14]{Poonen2012Random}, but replacing the relevant arithmetic duality theorem by its $S$-integral variant \cite[Thm.~I.4.10(b)]{Milne2006Duality}.

    Under the above assumptions\footnote{The reference also assumes that $(\sum q_v)(L) = 0$ but the proof does not use this assumption.} we can apply \cite[Lem.~2]{Skorobogatov2005Descent}. It implies that there exists a maximally isotropic subspace $K_v \subseteq V_v$ for each $v \in S$ such that $L \oplus \bigoplus_v K_v  = \bigoplus_v V_v$. We may thus take $a \in H^1(U, A)$ which maps to the unique  $\ell \in L$ such that $\ell + \beta \in \bigoplus_v K_v$. 
\end{proof}
\begin{proof}[Proof of Proposition \ref{prop:signature_(2,2,2)}]
    Take $a$ as in Lemma \ref{lem:existence_of_good_twist} and consider the twisted torsor $\pi_a: \mathcal{Y}_a \to \mathcal{X}$ and let $Y_a$ be the generic fiber of $\mathcal{Y}_a$.

    If $v \in S$ then we have $Y_a \cong C_{a + \beta}$. Since $q_v(a + \beta) = 0$ by construction it follows from the description of $q_v$ that $(Y_a)_{k_v}$ is a split Brauer-Severi curve over $k_v$ so $Y_a(k_v) \neq \emptyset$.

    If $v \not \in S$ then $(\mathcal{Y}_a)_{\mathcal{O}_v}$ is a smooth proper genus $0$ curve over $\mathcal{O}_v$ so $Y_a(k_v) \neq \emptyset$.

    By the Hasse principle for genus $0$ curves and the valuative criterion for properness there thus exists a point $y \in \mathcal{Y}_a(U) = Y_a(k)$. We then have $\pi_a(y) \in \mathcal{X}(U)$.
\end{proof}
\subsection{The elementary obstruction and \texorpdfstring{$\Beh(\mathcal{X})$}{Beh(X)}}
The goal of this subsection is to extend \cite[Prop. 3.3]{Harari2008Local-global}, which relates the Brauer-Manin obstruction and the elementary obstruction, to integral models of Deligne-Mumford stacks. This will allow us to prove Theorem \ref{The elementary obstruction is the only obstruction}.(\ref{The Brauer-Manin obstruction coming from beh is the only one}). We use the same notation as \S\ref{sec:elementary_obstruction} except that we let $\mathcal{X}$ be any regular Deligne-Mumford stack and $\pi: \mathcal{X} \to U$ any morphism (not necessarily smooth or proper).

We recall the definition of the group $\Beh(\mathcal{X}) := \ker( \Br_a \mathcal{X} \to \prod_{v \in S} \Br_a \mathcal{X}_{k_v})$ of locally constant algebras.

The evaluation map at a place $v$
\begin{equation*}
    \text{ev}_v(\alpha; \cdot): X(k_v) \text{ if  } v \in S, \mathcal{X}(\mathcal{O}_v) \text{ otherwise} \to \Q/\Z: x_v \to \text{inv}_v(\alpha(x_v)).
\end{equation*}
is constant for $\alpha \in \Beh(\mathcal{X})$. If $v \in S$ this is because the image of $\alpha$ in $\Br X_{k_v}$ lies in $\Br_0 X_{k_v}$. If $v \not \in S$ then its image is an element of $\Br \mathcal{X}_{\mathcal{O}_v}$ so $\alpha(x_v) \in \Br \mathcal{O}_v = 0$.

It follows that if $\mathcal{X}(\A_U) \neq \emptyset$ then for any choice $(x_v)_{v \in S} \in \prod_{v \in S} \mathcal{X}(k_v)$ the Brauer-Manin pairing $\mathcal{X}(\A_U) \times \Beh(\mathcal{X}) \to \Q/ \Z$ factors through the map
\begin{equation}
    \beh: \Beh(\mathcal{X}) \to \Q/\Z: \alpha \to \sum_{v \in S} \text{inv}_v(\alpha(x_v)).
    \label{Definition of the map beh}
\end{equation}

We can also describe $\beh$ as follows, analogously to \cite[Lemma 3.1]{Harari2008Local-global}.
\begin{lemma}
If $\mathcal{X}(k_v) \neq \emptyset$ for all $v \in S$ then the following commutative diagram has exact rows
\begin{equation}
    \begin{tikzcd}
	0 & {\Br U} & {\Br_1 \mathcal{X}} & {\Br_a \mathcal{X}} & 0 \\
	0 & {\bigoplus_{v \in S} \Br k_v} & {\bigoplus_{v \in S} \Br_1 \mathcal{X}_{k_v}} & {\bigoplus_{v \in S} \Br_a \mathcal{X}_{k_v}} & 0
	\arrow[from=1-2, to=1-3]
	\arrow[from=1-2, to=2-2]
	\arrow[from=2-2, to=2-3]
	\arrow[from=1-3, to=2-3]
	\arrow[from=1-3, to=1-4]
	\arrow[from=2-3, to=2-4]
	\arrow[from=2-1, to=2-2]
	\arrow[from=1-1, to=1-2]
	\arrow[from=1-4, to=1-5]
	\arrow[from=2-4, to=2-5]
	\arrow[from=1-4, to=2-4]
\end{tikzcd}
\label{Diagram of Brauer groups to which we can apply the snake lemma}
\end{equation}
The map $\beh$ is equal to the boundary map $\Beh(\mathcal{X}) \to \emph{coker}(\Br U \to \bigoplus_{v \in S} \Br k_v)$ coming from the snake lemma composed with the map $\bigoplus_{v \in S} \Br k_v \to \Q/ \Z: (\alpha_v)_v \to  \sum_{v \in S} \emph{inv}_v(\alpha_v)$.
\label{The map beh coming from the snake lemma}
\end{lemma}
\begin{proof}
Right exactness of the rows is true by the definition of $\Br_a \mathcal{X}$. For $v \in S$ we can take $x_v \in \mathcal{X}(k_v)$ which defines a section $\Br_1 \mathcal{X}_{k_v} \to \Br k_v$ by functoriality of the Brauer group. The fact that $\Br U$ injects into $\bigoplus_{v \in S} \Br k_v$ implies that the top row is also left exact. 

Using the construction of the snake lemma we see that the image of $\alpha \in \Beh (\mathcal{X})$ under the boundary map is $(\alpha(x_v))_{v \in S} \in \text{coker}(\Br U \to \bigoplus_{v \in S} \Br k_v)$ for any choice of $(x_v)_{v \in S} \in \prod_{v \in S} \mathcal{X}(k_v)$. The lemma follows immediately.
\end{proof}

We require the notion of arithmetic \'etale cohomology with compact support $H^p_c(U, \cdot)$. If $k$ is a function field over a finite field then this is the usual cohomology with compact support. If $k$ is a number field then one has to make certain modifications to deal with the archimedean places. We refer to \cite[p. 165]{Milne2006Duality} for the details. We denote by $R\Gamma_c(U, \cdot) \in \mathcal{D}(\text{Ab})$ the complex of abelian groups computing $H^p_c(U, \cdot)$, i.e. what is referred to as $H_c(U, \cdot)$ in 
\cite[p. 166]{Milne2006Duality}.

The defining property \cite[Prop. II.2.3]{Milne2006Duality} of this cohomology with compact support is the existence of the following distinguished triangle in $\mathcal{D}(\text{Ab})$
\begin{equation*}
    R\Gamma_c(U, \mathcal{F}) \to R\Gamma(U, \mathcal{F}) \to \bigoplus_{v \in S} R \Gamma(k_v, j_v^* \mathcal{F}) \to R\Gamma_c(U, \mathcal{F})[1]
\end{equation*}
for any complex $\mathcal{F} \in \mathcal{D}(U)$. Here $j_v: \Spec k_v \to U$ is the map coming from the inclusion $k \subset k_v$ and $R\Gamma(k_v, \cdot)$ is the complex computing Galois or Tate cohomology if $v$ is non-archimedean or archimedean, respectively. 

The long exact sequence corresponding to this triangle is
\begin{equation}
    \cdots \to H^p_c(U, \mathcal{F}) \to H^p(U, \mathcal{F}) \to \bigoplus_{v \in S} H^p(k_v, j_v^{*}\mathcal{F}) \to H^{p + 1}_c(U, \mathcal{F}) \to \cdots 
    \label{Long exact sequence of compactly supported cohomology}
\end{equation}
where $H^p(k_v, \cdot)$ is the Galois or Tate hypercohomology if $v$ is non-archimedean or archimedean, respectively.

If $S$ contains at least one non-archimedean place then $H^3(U, \G_m) = 0$ by \cite[Rem. II.2.2]{Milne2006Duality}. So Lemma \ref{Cohomology of KD and the Brauer group} and the above long exact sequence imply that
\begin{equation}
    \Beh(\mathcal{X}) \cong \text{im}( H^1_c(U, KD'(\mathcal{X})) \to H^1(U,KD'(\mathcal{X}))).
    \label{Surjectivity of H^1_c to Beh}
\end{equation}
\begin{proposition}
If $\mathcal{X}(k_v) \neq \emptyset$ for all $v \in S$ and $S$ contains at least one non-archimedean place then the following diagram commutes.
\begin{equation*}
    \begin{tikzcd}
	{H^1_c(U, KD'(\mathcal{X})) } & {H^1_c(U, \G_m[2]) } & {H^3_c(U, \G_m) } \\
	{\Beh(\mathcal{X})} && {\Q/\Z} \\
	&& {}
	\arrow["{w_{\mathcal{X}*}}", from=1-1, to=1-2]
	\arrow["{=}"{description}, draw=none, from=1-2, to=1-3]
	\arrow["\cong"', from=1-3, to=2-3]
	\arrow[from=1-1, to=2-1]
	\arrow["\beh", from=2-1, to=2-3]
\end{tikzcd}
\end{equation*}

Here $w_{\mathcal{X}*}$ is induced by $w_{\mathcal{X}}$ in the triangle \eqref{The fundamental triangle of KD} and $H^3_c(U, \G_m) \cong \Q/\Z$ is the isomorphism of \cite[Prop II.2.6]{Milne2006Duality}.
\label{Relation between beh and the boundary map of the triangle}
\end{proposition}
\begin{proof}
This is basically contained in the proof of \cite[Prop. 3.3]{Harari2008Local-global}. For the sake of clarity we give more details.

The distinguished triangle \eqref{The fundamental triangle of KD} induces the following distinguished triangle 
\begin{equation}
    R\Gamma_c(U, \G_m)[1] \to R\Gamma_c(U, KD(\mathcal{X})) \to  R\Gamma_c(U, KD'(\mathcal{X})) \xrightarrow{w_{\mathcal{X}*}} R\Gamma_c(U, \G_m)[2].
    \label{Triangle in compactly supported cohomology}
\end{equation}

Consider the following sequence of complexes, where the first row has cohomological degree $1$
\begin{equation*}
\begin{tikzcd}
	{H^2(U, \G_m)} & {H^1(U, KD(\mathcal{X}))} & {H^1(U, KD'(\mathcal{X}))} \\
	{\bigoplus_{v \in S}H^2(k_v, \G_m)} & {\bigoplus_{v \in S} H^1(k_v, j_v^*KD(\mathcal{X}))} & {\bigoplus_{v \in S} H^1(k_v, j_v^*KD'(\mathcal{X}))} 
	\arrow[from=1-1, to=2-1]
	\arrow[from=1-2, to=2-2]
	\arrow[from=1-3, to=2-3]
	\arrow[from=1-1, to=1-2]
	\arrow[from=1-2, to=1-3]
	\arrow[from=2-1, to=2-2]
	\arrow[from=2-2, to=2-3]
\end{tikzcd}
\end{equation*}
Using Lemma \ref{Cohomology of KD and the Brauer group} we can identify this with the short exact sequence of complexes \eqref{Diagram of Brauer groups to which we can apply the snake lemma}. It is in particular a short exact sequence of complexes and thus induces a distinguished triangle $T$ in $\mathcal{D}(\text{Ab})$.

The complex $R\Gamma_c(U, \cdot)$ is by definition a shift of the mapping cone of $R \Gamma(U, \cdot) \to \bigoplus_{v \in S} R \Gamma(k_v, j_v^* \cdot)$. By truncating on both sides we thus get a zigzag of maps from the triangle \eqref{Triangle in compactly supported cohomology} to the triangle $T$ in $\mathcal{D}(\text{Ab})$, inducing a zigzag of maps between the corresponding long exact sequences.

We now note that the map $H_c^1(U, KD'(\mathcal{X})) \to \Beh(\mathcal{X})$ is the one induced by this zigzag in cohomology. Similarly, the isomorphism $H^3_c(U, \G_m) \cong \Q/\Z$ is by definition \cite[Prop. II.2.6]{Milne2006Duality} given by composing the isomorphism $H^3_c(U, \G_m) \cong \text{coker}(\Br U \to \oplus_v \Br k_v)$ coming from the zigzag with the isomorphism $\text{coker}(\Br U \to \oplus_v \Br k_v) \cong \Q/\Z: (\alpha_v)_v \to \sum_v \text{inv}_v(\alpha_v)$.

The boundary map in the long exact sequence coming from the triangle \eqref{Triangle in compactly supported cohomology}, respectively the triangle $T$, is $w_{\mathcal{X}*}$, respectively the one coming from the snake lemma. The proposition then follows from Lemma \ref{The map beh coming from the snake lemma}.
\end{proof}
We can now prove Theorem \ref{The elementary obstruction is the only obstruction}.(\ref{The Brauer-Manin obstruction coming from beh is the only one}).
\begin{theorem}
Let $U$ be a finite type Dedekind scheme whose fraction field $k$ is a global field and such that at least one non-archimedean place of $k$ does not define a point of $U$.

Let $\mathcal{X} \to U$ be a tame relative stacky curve whose generic fiber $X$ has genus $g < 1$. If $\mathcal{X}(\A_U)^{\Beh(\mathcal{X})} \neq \emptyset$ then $\mathcal{X}(U) \neq \emptyset$.
\label{The locally constant algebras give the only obstruction}
\end{theorem}
\begin{proof}
The complex $KD'(\mathcal{X})$ is a locally constant finitely generated sheaf due to Lemma \ref{R pi G_m of a stacky curve} so the following pairing 
\begin{equation}
    H^1_c(U, KD'(\mathcal{X})) \times \text{Ext}^2(KD'(\mathcal{X}), \G_m) \to H^3_c(U, \G_m) \cong \Q/\Z: (\alpha, w) \to w_*(\alpha).
    \label{Perfect pairing in etale cohomology}
\end{equation}
induces an injective map $\text{Ext}^2(KD'(\mathcal{X}), \G_m) \to \Hom( H^1_c(U, KD'(\mathcal{X})), \Q/\Z)$ by \cite[Thm. II.3.1]{Milne2006Duality}.

The assumption $\mathcal{X}(\A_U)^{\Beh(\mathcal{X})} \neq \emptyset$ implies that $\beh = 0$ and thus by Proposition \ref{Relation between beh and the boundary map of the triangle} that the map $w_{\mathcal{X} *}:H^1_c(U, KD'(\mathcal{X})) \to H^3_c(U, \G_m)$ is zero. The injectivity then implies that $w_{\mathcal{X}} = 0$. The theorem now follows from Theorem \ref{The elementary obstruction is the only obstruction if genus < 1}.
\end{proof}
\section{Genus \texorpdfstring{$\frac{1}{2}$}{1/2} curves failing the Hasse principle}
Let $X$ be a genus $\frac{1}{2}$ tame stacky curve over a number field $k$. It follows from the formula \eqref{Genus of stacky curve} that it has signature $(0; 2,2)$. Assume that $X(\A_k) \neq \emptyset$, then $X_{\text{coarse}}(\A_k) \neq \emptyset$ and thus $X_{\text{coarse}} \cong \Proj^1_k$. The stacky locus $\mathcal{P} \subset \Proj^1$ is defined by a binary quadratic form $f\in k[X,Y]$. It follows from \cite[Lemma 5.3.10]{Voight2022Canonical} that $X \cong \Proj^1_k[\sqrt{f}]$ where $ \Proj^1_k[\sqrt{f}]$ is the stack of square-roots of $f$ \cite[Def. 2.2.1]{Cadman2007Tangency}. We may thus assume that $X = \Proj^1_k[\sqrt{f}]$. 

This stack has a universal cover given by $Y := \{W^2 = f(U,V) \}$ with the morphism $Y \to X$ being given by $[u:v:w] \to [u:v]$ on coarse moduli spaces, see Example \ref{The universal cover there are at most two stacky points} for a computation of this universal cover. This is a $\mu_2$-torsor, denote its class by $[Y/X] \in H^1(X, \mu_2)$. The Brauer group $\Br X$ is generated by $\Br k$ and the cup products $d \cup [Y/X] \in \Br X$ for $d \in k^{\times}/k^{\times 2} \cong H^1(k, \Z/2 \Z)$ by Propositions \ref{Brauer group is generated by cyclic algebras} and \ref{prop:Brauer_group_alg_closed_field}.

Assume that $f = a x^2 + bxy + c y^2 \in \mathcal{O}_k[x,y]$, this can always be achieved by a projective linear transformation, and let $q := b^2 - 4ac$ be the discriminant of $f$. We will now consider the integral model of $X$ given by the root stack $\mathcal{X} := \Proj_{\mathcal{O}_k[\frac{1}{2q}]}^1[\sqrt{f}]$. This is smooth over $\mathcal{O}_k[\frac{1}{2q}]$ because $\{f = 0 \} \subset \Proj^1_{\mathcal{O}_k[\frac{1}{2q}]}$ is smooth over $\mathcal{O}_k[\frac{1}{2q}]$. Let $X^{\circ} := \{f \neq 0\} \subset X$ be the maximal open subscheme. It follows from the definition of a root stack that $\mathcal{X}(\mathcal{O}_{k}[\frac{1}{2q}]) \cap X^{\circ}(k)$ is given by those $[x:y] \in X^{\circ}(k)$ such that $v(f(x,y))$ is even for all finite places $v \nmid 2q$ and that we have an analogous description for $\mathcal{O}_v$-points. We have $\mathcal{X}(\mathcal{O}_v) \neq \emptyset$ for all finite places $v \nmid 2q$ because $f$ is smooth over $\mathcal{O}_v$ so we can choose $x,y \in \mathcal{O}_v$ such that $v(f(x,y)) = 0$, where we use $\deg f < 4 \leq |\Proj^1(\F_v)|$. For the other places $v$ we have $X(k_v) \supset X^{\circ}(k_v) \neq \emptyset$.

We now specialize to the case $k = \Q$ and will compute when there is an elementary obstruction. 
\begin{lemma}
We have 
\begin{equation*}
    \Beh(\mathcal{X}) \cong \Big\{d \in \Z[\frac{1}{2q}]^{\times}/\{1,q\} \cdot \Z[\frac{1}{2q}]^{\times 2} :d \in \{1, q\} \Q_p^{\times 2} \emph{ for } p \mid 2q \emph{ and } d > 0 \emph{ if } q > 0\Big\}.
\end{equation*}
\end{lemma}
\begin{proof}
The Dedekind scheme $U := \Spec \Z[\frac{1}{2q}]$ does not contain the non-archimedean prime $2$ so $H^3(U, \G_m) = 0$ by \cite[Rem. II.2.2]{Milne2006Duality} so by Lemma \ref{Cohomology of KD and the Brauer group}
\begin{equation}
    \Beh(\mathcal{X}) \cong \ker \Big(H^1(U, \Pic X_{\overline{\Q}}) \to \bigoplus_{v \mid 2q \infty} H^1(\Q_v, \Pic X_{\overline{\Q}_v})\Big).
    \label{Beh in terms of H^1(U, Pic)}
\end{equation}

It follows from Lemma \ref{Picard group of stacky curve with coarse genus 0} that $\Pic^0 X_{\overline{\Q}} \cong \Z/2\Z$ with the trivial $\Gal(\overline{\Q}/\Q)$-action. We deduce via Lemma \ref{Computation of H^1(U, Pic)} and Kummer theory that
\begin{equation*}
    H^1(U, \Pic X_{\overline{\Q}}) \cong H^1(U, \Z/2\Z)/ \varphi \cong (\Z[\frac{1}{2q}]^{\times}/ \Z[\frac{1}{2q}]^{\times 2})/q.
\end{equation*}
Here $\varphi: \pi_1(U) \to \Z/2\Z$ is the cocycle which sends $\sigma$ to $0$ if $\sigma$ fixes the roots of $f(x,y)$ and to $1$ otherwise. Its image under the isomorphism $H^1(U, \Z/2\Z) \cong \Z[\frac{1}{2q}]^{\times}/ \Z[\frac{1}{2q}]^{\times 2}$ is $q$.

Completely analogously we get $ H^1(\Q_v, \Pic X_{\overline{\Q}_v}) \cong \Q_v^{\times}/ \{1, q\} \cdot \Q_v^{\times 2}$ for $v \mid 2q \infty$. The lemma then follows by using these descriptions in \eqref{Beh in terms of H^1(U, Pic)}.
\end{proof}

An algebra corresponding to $d \in \Beh(\mathcal{X})$ is by Proposition \ref{Brauer group is generated by cyclic algebras} given by the cup product $d \cup [Y/X]$.

Recall that to compute the Brauer obstruction of $\Beh(\mathcal{X})$ we have to compute the map $\beh$ defined in \eqref{Definition of the map beh}. Recall \cite[IV,\S2.1]{Serre1973Arithmetic} that to each binary quadratic form $f$ and place $v$ we can assign its $\epsilon$-\emph{invariant} $\epsilon_v := \epsilon_v(f) \in \mu_2$. We will also let the evaluation map $\text{ev}_v(d; \cdot)$ take values in $\mu_2$ instead of the isomorphic group $\Z/2\Z$.

\begin{lemma}
Let $d \in \Beh(\mathcal{X})$. Let $v \mid 2q \infty$ be a place.
\begin{enumerate}
    \item If $d \in \Q_v^{\times 2}$ then we have $\emph{ev}_v(d ;P_v) = 1$ for all $P_v \in X(\Q_v)$.
    \item If $d \in q \Q_v^{\times 2}$ then we have $\emph{ev}_v(d; P_v) = \epsilon_v$ for all $P_v \in X(\Q_v)$.
\end{enumerate}
This implies that $\beh(d) = \prod_{v \mid 2q \infty, d \in q \Q^{\times 2}_v} \epsilon_v$.
\label{Local computations}
\end{lemma}

\begin{proof}
The evaluation map is locally constant so we can assume that $P_v \in X^{\circ}(k_v)$. We can thus use that $(d \cup [Y/X])|_U = (d, f/X^2) \in \Br X^{\circ}$ is a quaternion algebra.

For $P_v = [x:y] \in X^{\circ}(\Q_v)$ we have $\text{ev}_v(d ;P_v) = (d, f(x,y)/x^2)_v$. If $d \in \Q_v^{\times 2}$ then $(d, f(x,y)/x^2)_v = 1$ so assume that $d \in q \Q_v^{\times 2}$. In that case \cite[IV, \S2.2: Cor. of Thm. 6]{Serre1973Arithmetic} implies that $(d, f(x,y)/x^2)_v = (q, f(x,y)/x^2)_v = \epsilon_v$ (note that the convention for the discriminant differs by a sign).

The last statement follows from the definition of $\beh$.
\end{proof}

The considerations above and Theorem \ref{The elementary obstruction is the only obstruction} immediately lead to
\begin{theorem}
Let $f \in \Z[X,Y]$ be a binary quadratic form with discriminant $q \neq 0$ and $\epsilon$-invariants $\epsilon_v$. The root stack $\mathcal{X} = \Proj^1_{\Z[\frac{1}{2q}]}[\sqrt{f}]$ always has local integral points everywhere and has a $\Z[\frac{1}{2q}]$-point if and only if there exists no (possibly negative) $d \mid 2q$ such that:
\begin{enumerate}
    \item For all places $v \mid 2q \infty$ one has $d \in \Q_v^{\times 2}$ or $d \in q \Q_v^{\times 2}$.
    \item We have $\prod_{v \mid 2 q \infty, d \not \in \Q_v^{\times 2}} \epsilon_v = \prod_{v \mid 2 q \infty, d \in \Q_v^{\times 2}} \epsilon_v = -1 $, where the first equality always holds.
\end{enumerate}
\label{Criterion for a Brauer-Manin obstruction coming from a single algebra}
\end{theorem}
\begin{proof}
The equality $\prod_{v \mid 2 q \infty, d \not \in \Q_v^{\times 2}} \epsilon_v = \prod_{v \mid 2 q \infty, d \in \Q_v^{\times 2}} \epsilon_v$ follows from the product formula $\prod_v \epsilon_v = 1$ and the fact that $\epsilon_v = 1$ for $v \nmid 2 q \infty$ since $f$ has good reduction there.

Everything else follows from Lemma \ref{Local computations}, Theorem \ref{The locally constant algebras give the only obstruction} and the preceding discussion.
\end{proof}

\begin{example}
A simple example is $f = 3x^2 + 2xy + 5y^2$ for which $q = -56 = -7\cdot 8$. Completing the square we find $\epsilon_v = (3, 42)_v$, so $\epsilon_{\infty} = 1, \epsilon_{2} = \epsilon_{7} = -1$. We can thus apply Theorem \ref{Criterion for a Brauer-Manin obstruction coming from a single algebra} with $d := -7$. Indeed, $q < 0$ so $d \in q \R^{\times 2}$, $d \in \Z_2^{\times 2}$ and $d \in  - 56 \Z_7^{\times 2}$ by Hensel's lemma. We conclude that $\mathcal{X}(\Z[\frac{1}{14}]) = \emptyset$.
\end{example}

We can also reinterpret the example of Bhargava-Poonen \cite[Thm. 1]{Bhargava2020Stacky}. Note that unlike them we do not assume that $p$ is a square mod $q$ nor that $a$ is a non-square modulo $p$ and a square modulo $q$.
\begin{example}
Let $p, q, r$ be primes congruent to $7 \pmod{8}$ such that $p$ and $q$ are squares modulo $r$. Let $f = a x^2 + b xy + c y^2$ be a positive definite integral binary quadratic form of discriminant $-pqr$ such that $a$ is a non-square modulo $r$. An example is $f = 3 x^2 + x y + 850 y^2$ with $p = 7, q= 47,r = 31$.

By completing the square one sees that $\epsilon_v = (a, a pqr)_v = (a, - pqr)_v$. Since $- pqr \equiv 1 \pmod{8}$, and by quadratic reciprocity we get that $\epsilon_{\infty} = \epsilon_{2} = 1, \epsilon_{r} = -1$. We can then apply Theorem \ref{Criterion for a Brauer-Manin obstruction coming from a single algebra} to $d := -r$. We have $-r \in \Z_2^{\times 2}, \Z_{p}^{\times 2}, \Z_{q}^{\times 2}$ and $-r \in -pqr\Z_r^{\times 2}, - pqr \R^{\times 2}$. We conclude that $\mathcal{X}(\Z[\frac{1}{2pqr}]) = \emptyset$.
\end{example}
\bibliographystyle{alpha} 
\bibliography{references}
\end{document}